\documentclass[10pt]{siamltex}
\usepackage{epsfig}
\usepackage{amssymb}
\usepackage{amsbsy}
\usepackage{amsmath}
\usepackage{enumerate}
\usepackage{float}
\usepackage{graphicx}
\usepackage{xspace}

\usepackage{bm}
\usepackage{color}

\usepackage{remarks}
\usepackage{def} 
\usepackage[numbers]{natbib}
\let\cite=\citet

\begin{document}

\newcommand\footnotemarkfromtitle[1]{%
\renewcommand{\thefootnote}{\fnsymbol{footnote}}%
\footnotemark[#1]%
\renewcommand{\thefootnote}{\arabic{footnote}}}

\title{Finite element quasi-interpolation and best approximation\footnotemark[1]}
\author{ Alexandre Ern\footnotemark[2] \and Jean-Luc Guermond\footnotemark[3]}

\date{Draft version \today}

\maketitle

\renewcommand{\thefootnote}{\fnsymbol{footnote}} \footnotetext[1]{
  This material is based upon work supported in part by the National
  Science Foundation grant DMS-1217262, by the Air
  Force Office of Scientific Research, USAF, under grant/contract
  number FA9550-15-1-0257 and by the Army Research Office under grant
  number W911NF-15-1-0517,
Draft
  version, \today}
\footnotetext[2]{Universit\'e Paris-Est, CERMICS (ENPC),
  77455 Marne-la-Vall\'ee cedex 2, France.}
\footnotetext[3]{Department of Mathematics, Texas
  A\&M University 3368 TAMU, College Station, TX 77843, USA.}

\renewcommand{\thefootnote}{\arabic{footnote}}

\begin{abstract} 
  This paper introduces a quasi-interpolation operator for scalar- and
  vector-valued finite element spaces constructed on affine,
  shape-regular meshes with some continuity across mesh interfaces.
  This operator gives optimal estimates of the best approximation
    error in any $L^p$-norm assuming regularity in the fractional
    Sobolev spaces $W^{r,p}$, where $p\in [1,\infty]$ and the smoothness index $r$ can be
    arbitrarily close to zero.  The operator is stable in $L^1$,
  leaves the corresponding finite element space point-wise invariant,
  and can be modified to handle homogeneous boundary conditions.  The theory is
  illustrated on $H^1$-, $\bH(\text{curl})$- and
  $\bH(\text{div})$-conforming spaces.
\end{abstract}



\pagestyle{myheadings} \thispagestyle{plain} 
\markboth{A. ERN, J.L. GUERMOND}{Finite element quasi-interpolations and best approximation}

\section{Introduction} 
Consider a shape-regular sequence of affine meshes $\famTh$
approximating a bounded polyhedral Lipschitz domain $\Dom$ in $\Real^d$, and a
sequence of finite element spaces $(P(\calT_h))_{h>0} $ based on this
mesh sequence, composed of either scalar- or vector-valued functions,
and conforming in some functional space $W$
where some continuity across mesh interfaces is
enforced. Estimates of the best-approximation error in
$P(\calT_h)$ are invariably invoked in the convergence analysis of
finite element approximations. When the exact solution is smooth
enough, the canonical interpolation operator in $P(\calT_h)$ can be
used to obtain decay estimates of the best-approximation error in
terms of the mesh-size. However, in many practical situations, the
exact solution only sits in a Sobolev space $W^{r,p}(\Dom)$ (for
some $p\in [1,\infty]$) where the smoothness index $r$ can be close
to zero. In this case, an alternative quasi-interpolation operator
must be invoked to estimate the best-approximation error.

The aim of this paper is the construction of a quasi-interpolation
operator that is stable in $L^1$, leaves $P(\calT_h)$
point-wise invariant, and approximates (quasi-)locally and optimally
functions in $W^{r,p}(\Dom)$, for all $p\in [1,\infty]$ and $r$
arbitrarily close to zero. Moreover, the construction can be modified
to enforce homogeneous boundary conditions that are legitimate in
$W$. The main examples we have in mind are $H^1$-,
$\bH(\text{curl})$-, and $\bH(\text{div})$-conforming finite element
spaces. Let us emphasize that, for vector-valued elements, the present
construction is not a substitute to the notion of commuting bounded
cochain projection introduced in the framework of Finite Element
Exterior Calculus, but instead a vital complement to it.  Indeed, as
shown in \cite[Theorem~5.6]{Arnold_Falk_Whinter_2006}, the
approximation error for the commuting projection is bounded by the
best-approximation error in $P(\calT_h)$, up to a uniform constant.
Therefore, the present quasi-interpolation operator shows that
commuting bounded cochain projections converge optimally for
low-regularity solutions. 
Another important consequence of the present construction in the case of
vector-valued elements is that the decay rates of the
best-approximation error only involve the solution seminorm in
$W^{r,p}(\Dom)$ without the need to invoke a bound on the curl or the
divergence of the solution.

Let us put our results in perspective with the literature.  The
construction of quasi-interpolation operators with the above
properties for $H^1$-conforming finite elements is a well-studied
field.  In two space dimensions on triangular meshes, we mention the
early work of \cite{Cl75}; see also \cite{BerGi:98} where the
construction is modified so that the resulting operator leaves the finite element space point-wise
invariant. In a landmark paper by \cite{ScoZh:90}, an alternative
construction is proposed in any space dimension for $H^1$-conforming
finite elements with and without homogeneous boundary conditions; see
also \cite{Girault_Lions_2001} (and \citep[p.~491]{ScoZh:90}) for a
modification leading to $L^1$-stability.  The so-called Scott--Zhang
interpolation operator leads to optimal decay estimates of the
approximation error for functions in $H^1(\Dom)$.  The extension of
this result to functions in $W^{r,p}(\Dom)$ with $r\in(\frac1p,1)$ has
been studied only recently by~\cite{Ciarlet:13} using the
original Scott--Zhang operator. To the authors' best knowledge,
in contrast to $H^1$-conforming elements, quasi-interpolation
results for $\bH(\text{curl})$- and $\bH(\text{div})$-conforming
elements are missing in the literature.

The main ingredient for the Cl\'ement and Scott--Zhang construction
is a regularization of functions based on macroelements consisting
of patches of elements. We adopt here a somewhat different
route.  The key idea is a two-step construction, namely a
projection onto the (fully discontinuous) broken finite element
space followed by an averaging operator that stitches the result
in the spirit of~\cite[Eqs.~(25)-(26)]{Oswald_Comput_1993}.  This
way, patches of mesh cells are only involved when handling
discrete objects, thereby avoiding the delicate question of
reference patches frequently used in the literature.
Averaging operators are often invoked in the literature in various
contexts; we mention, in particular, a posteriori error
estimates~\citep{AcBeC:03,KarPa:03,ErnSV:10},
preconditioners~\citep{Schoberl_Lehren:13}, stabilization
techniques~\citep{BurEr:07}, and discontinuous Galerkin methods with
improved stability properties~\citep{CoKaS:07,CamSo:15}. The
novel feature here is the analysis of averaging operators in the
vector-valued case and the application to quasi-interpolation.  We
emphasize that we devise a unifying framework for the analysis
that works irrespective of the nature of the degrees of freedom,
\ie whether they are nodal values, integrals over edges, faces, or
cells.  Essentially all the finite elements that are traditionally 
used to build $H^1$-, $\bH(\text{curl})$-, and $\bH(\text{div})$-conforming
finite element spaces match the few assumptions of the unified
analysis.

The paper is organized as follows.  In \S\ref{Sec:Finite_Elements}, we
introduce the notation and construct a sequence of abstract finite
element spaces conforming in some functional space $W$. Key
assumptions are identified and listed. 
A local interpolation operator stable in $L^1$
is constructed in \S\ref{Sec:local_approximation}. 
An abstract averaging operator acting
only on discrete functions is introduced and analyzed in
\S\ref{Sec:averaging_operator}. The final
quasi-interpolation operator is constructed in
\S\ref{Sec:quasi_interpolation_operator} without enforcing any
boundary condition. 
The question of 
enforcing boundary conditions is addressed in
\S\ref{Sec:boundary_conditions}.
Finally, \S\ref{sec:technical} contains two technical results
on fractional Sobolev spaces that are of independent interest:
a Poincar\'e inequality and a trace inequality.

\section{Finite elements}\label{Sec:Finite_Elements}
In this section we introduce some notation and construct a sequence of
abstract finite element spaces, conforming in some functional space
$W$. In the entire paper the space dimension is denoted by $d$ and the domain
$D$ is a bounded Lipschitz polyhedron in $\Real^d$.

\subsection{Meshes}
Let $\famTh$ be a mesh sequence that we assume to be
affine and shape-regular in the sense of Ciarlet.  We also assume for
the sake of simplicity that the meshes cover $\Dom$ exactly, that they are composed of convex cells, and that
they are matching, \ie 
for all cells $K,K'\in\calT_h$ such that
$K\ne K'$ and $K\cap K'\ne\emptyset$, the set $K\cap K'$ is a common
vertex, edge, or face of both $K$ and $K'$ (with obvious extensions in
higher space dimensions).  By convention, given a mesh $\calT_h$, the
elements $K\in \calT_h$ are closed sets in $\Real^d$.  

We assume that
there is a reference element $\wK$ such that for any mesh $\calT_h$
and any cell $K\in \calT_h$, there is an affine bijective map
between $\wK$ and $K$, which we henceforth denote
$\trans_K :\wK \to K$. 
Since $\trans_K$ is affine and bijective, there is an invertible
matrix $\Jac_K\in\Real^{d\CROSS d}$ such that
\begin{equation}
\trans_K(\wbx)-\trans_K(\wby) = \Jac_K (\wbx-\wby), \qquad \forall \wbx,\wby\in \wK.
\label{Eq2:TransAff}
\end{equation}
In what follows, we denote points in $\Real^d$, $\Real^d$-valued
functions and $\Real^d$-valued maps in boldface type, and we denote
the Euclidean norm in $\Real^d$ by
$\|\SCAL\|_{\ell^2(\Real^d)}$, or $\|\SCAL\|_{\ell^2}$ when the
context is unambiguous. We abuse the notation by using the same symbol
for the induced matrix norm. Owing to the
shape-regularity assumption of the mesh sequence, 
there are uniform constants $c^\sharp$, $c^\flat$ such
that
\begin{equation}
| {\det(\Jac_K)} | = \mes{K}\mes{\wK}^{-1}, \qquad 
\| \Jac_K \|_{\ell^2} \leq c^\sharp h_K, \qquad 
\| \Jac_K^{-1}  \|_{\ell^2} \leq c^\flat h_K^{-1},
\label{Eq2:propJK}
\end{equation}
where $h_K$ is the diameter of $K$. Recall that
$c^\sharp=\frac{1}{\rho_{\wK}}$ and
$c^\flat = \frac{h_K}{\rho_K}h_{\wK}$ for meshes composed of
simplices, where $\rho_K$ is the diameter of the largest ball that can
be inscribed in $K$, $h_\wK$ is the diameter of $\wK$, and $\rho_\wK$
is the diameter of the largest ball that can be inscribed in $\wK$.

\subsection{Finite element generation}

We are going to consider various approximation spaces based on the
mesh sequence $\famTh$. Again for the sake of simplicity, we assume
that each approximation space is constructed from a fixed reference
finite element $\wKPS$.  The reference degrees of freedom are denoted
$\{\wsigma_{1},\ldots, \wsigma_{\nf}\}$ and the associated reference
shape functions are denoted $\{\wtheta_{1},\ldots, \wtheta_{\nf}\}$;
by definition $\wsigma_i(\wtheta_{j})=\delta_{ij}$, $\forall
i,j\in\intset{1}{\nf}$. We denote $\calN:=\intset{1}{\nf}$ to
alleviate the notation. The shape functions are $\Real^q$-valued for
some integer $q\ge1$. We henceforth assume that $\wP\subset
W^{1,\infty}(\wK;\Real^q)$ (recall that $\wP$ is typically
a space of polynomial
functions, but we do not require this assumption here).

We assume that there exists a Banach space
$V(\wK) \subset L^1(\wK;\Real^q)$ such that the linear forms
$\{\wsigma_{1},\ldots, \wsigma_{\nf}\}$ can be extended to
$\calL(V(\wK);\Real)$, \ie $V(\wK)$ is the domain of the degrees of
freedom; see~\citep[p.~39]{ErnGuermond_FEM}.  Then, we define $\inter_{\wK}:V(\wK)\to \wP$, the
interpolation operator associated with the reference finite element $\wKPS$, by
\begin{equation}
\inter_{\wK}(\wv)(\wbx) = 
\sum_{i\in \calN} \wsigma_i(v) \wtheta_i(\wbx), \qquad \forall \wbx\in
\wK,
\quad \forall \wv\in V(\wK).
\label{Eq2:OpIntLoc}
\end{equation}
By construction,
$\inter_\wK\in\calL(V(\wK);\wP)$, and $\wP$ is point-wise invariant under
$\inter_{\wK}$. 

We now address the question of constructing finite elements for any
mesh cell $K\in \calT_h$. We 
assume that there exists a Banach space $V(K)\subset L^1(K;\Real^q)$ and
a bounded, bijective, linear map between $V(K)$ and $V(\wK)$: 
\begin{equation} 
\label{eq:hyp_psiK} \mapK:
  V(K)\ni v\;\longmapsto\; \mapK(v)\in V(\wK). 
\end{equation}   
We then set
\begin{subequations}\label{Eq2:genEF}
\begin{align}
&P_K:=\bset  p=\mapKmun(\wp) \tq \wp \in \wP \eset,\label{Eq2:genEF_P}\\
&\Sigma_K:= \{\sigma_{K,i}\}_{i\in\calN} \; \text{s.t.} \;
\sigma_{K,i} = \wsigma_i\circ \mapK.\label{Eq2:genEF_S}
\end{align}
\end{subequations}
\begin{proposition}[Finite element]
\label{Prop:GenEF}
The triple $(K,P_K,\Sigma_K)$ is a finite element.
\end{proposition}

\begin{proof} 
  Note first that $\dim(P_K)=\dim(\wP)=\nf$ since $\psi_K$ is
  bijective. Moreover, a function $p\in P_K$ such that
  $\sigma_{K,i}(p)=0$ for all $i\in\calN$ is such that $\psi_K(p)=0$ by
  the unisolvence property of the reference finite element; hence,
  $p=0$. Finally, the linear forms $\sigma_{K,i}$ are in
  $\calL(V(K);\Real)$ since $|\sigma_{K,i}(v)| \le
  \|\wsigma_i\|_{\calL(V(\wK);\Real)} \|\psi_K\|_{\calL(V(K);V(\wK))}
  \|v\|_{V(K)}$, for all $v\in V(K)$.
\end{proof}

The above definitions lead us to consider the canonical interpolation
operator associated with the finite element $(K,P_K,\Sigma_K)$:
\begin{equation}
  \inter_{K}(v)(\bx) = 
  \sum_{i\in\calN} \sigma_{K,i}(v) \theta_{K,i}(\bx), \qquad \forall \bx\in K,\quad \forall v\in V(K),
\label{def_of_interK}
\end{equation}
where we have set
$\theta_{K,i} := \mapKmun(\wtheta_i)$. Note that
$\inter_K\in\calL(V(K);P_K)$ and that $P_K$ is point-wise invariant
under $\inter_K$.

Since the mesh is affine, we assume that $\mapK$ has a simple
structure; more precisely, we assume that there is a $q\CROSS q$
invertible matrix $\polA_K$ such that
\begin{equation} \label{localization_of_mapK}
\mapK(v) = \polA_K (v\circ\trans_K),
\end{equation}
which implies that $\mapK$ can be extended as a map from
$L^1(K;\Real^q)$ to $L^1(\wK;\Real^q)$. The following classical result
shows that $\mapK$ maps $W^{l,p}(K;\Real^q)$ to $W^{l,p}(\wK;\Real^q)$
for all $l\in\Natural$ and all $p\in [1,\infty]$. In particular, this
implies that $P_K\subset W^{1,\infty}(K;\Real^q)$. We use the
notation
$|\mapK|_{\calL(W^{l,p}(K;\Real^q);W^{l,p}(\wK;\Real^q))} :=
\sup_{v\in W^{l,p}(K;\Real^q)}
\frac{|\mapK(v)|_{W^{l,p}(\wK;\Real^q)}}{|v|_{W^{l,p}(K;\Real^q)}}$
and similar notation for
$|\mapKmun|_{\calL(W^{l,p}(\wK;\Real^q);W^{l,p}(K;\Real^q))}$,
where it henceforth implicitly understood the denominator
is not zero.

\begin{lemma}[Bound in Sobolev norms]
  Let $l\in\Natural$.
  There is a uniform constant $c$ depending on the shape-regularity of
  the mesh sequence $\famTh$ and on $l$ such that the following holds:
\begin{subequations}\label{eq:bnd_calL_psi}\begin{align}
|\mapK|_{\calL(W^{l,p}(K;\Real^q);W^{l,p}(\wK;\Real^q))} 
&\le c\, \|\polA_K\|_{\ell^2}\; \| \Jac_K \|_{\ell^2}^l  \; |{\det(\Jac_K)}|^{-\frac1p},\\
|\mapKmun|_{\calL(W^{l,p}(\wK;\Real^q);W^{l,p}(K;\Real^q))} 
&\le c\, \|\polA_K^{-1}\|_{\ell^2}\;\| \Jac_K^{-1} \|_{\ell^2}^l  \; |{\det(\Jac_K)}|^{\frac1p},\label{eq:bnd_calL_psi_mun}
\end{align}\end{subequations}
for all $K\in\calT_h$ and all $p\in [1,\infty]$ 
(with $z^{\pm \frac1p}=1$, $\forall z>0$ if $p=\infty$).
\end{lemma}
\begin{proof}
For any multilinear map
  $A\in\multilin_l(\Real^d,\ldots,\Real^d;\Real^q)$, let us set
\[
\|A\|_{\multilin_l(\Real^d,\ldots,\Real^d;\Real^q)}
:= \sup_{(\by_1,\ldots,\by_l)\in \Real^d\CROSS\ldots\CROSS\Real^d}
\frac{\|A(\by_1,\ldots,\by_l)\|_{\ell^2}}{\|\by_1\|_{\ell^2}\ldots
  \|\by_l\|_{\ell^2}}.
\]
Then, denoting by $D^l \mapK(v)$ the $l$-order Frechet derivative of
$\mapK$ at $v$, the assumption \eqref{localization_of_mapK} implies
that
$\|D^l \mapK(v)\|_{\multilin_l(\Real^d,\ldots,\Real^d;\Real^q)} \le
\|\polA_K\|_{\ell^2} \|D^l
(v\circ\trans_K)\|_{\multilin_l(\Real^d,\ldots,\Real^d;\Real^q)}$
for all $l\in \polN$. Then, standard results about the transformation
of seminorms in the Sobolev space $W^{l,p}$ using the pullback by
$\trans_K$ lead to~\eqref{eq:bnd_calL_psi}, see \eg
\cite[Thm.~3.1.2]{Ciarlet_FE_Book_2002} or \cite[Lemma
1.101]{ErnGuermond_FEM}.
\end{proof}

\begin{corollary}[Bound on $\polA_K$]
Assume that there is a uniform constant
$c$ so that
\begin{equation}
  \|\polA_K\|_{\ell^2}\|\polA_K^{-1}\|_{\ell^2}\le c\, \|\Jac_K\|_{\ell^2}\|\Jac_K^{-1}\|_{\ell^2}.
\label{Assumption:AkAkminusone_bounded}
\end{equation} 
Then the following holds for all $s,m\in\polN$, all $K\in\calT_h$ and all $p\in [1,\infty]$:
\begin{equation}\label{kappa_m_s_of_mapK}
|\mapKmun|_{\calL(W^{m,p}(\wK);W^{m,p}(K))} |\mapK|_{\calL(W^{s,p}(K);W^{s,p}(\wK))} 
\le c\, h_K^{s-m}.
\end{equation}
\end{corollary}

\begin{proof}
Combine \eqref{eq:bnd_calL_psi} with \eqref{Eq2:propJK} and
\eqref{Assumption:AkAkminusone_bounded}.
\end{proof}

\begin{lemma}[Norm equivalence] \label{Lem:Linfty_bounded_by_sigma} 
There exists a uniform constant $c$ such that
\begin{align} \label{eq:Linfty_ddl}
\|v\|_{L^\infty(K;\Real^q)} \le c\, \|\polA_K^{-1}\|_{\ell^2} \sum_{i\in \calN}
|\sigma_{K,i}(v)|, \\
\label{eq:ddl_Linfty} 
\sum_{i\in \calN}
|\sigma_{K,i}(v)| \le c\, \|\polA_K\|_{\ell^2} \|v\|_{L^\infty(K;\Real^q)},
\end{align}
for all $v \in P_K$ and all $K\in\calT_h$.
\end{lemma}
\begin{proof} 
We prove~\eqref{eq:Linfty_ddl}; the proof for the other 
bound is similar. We observe that
\begin{align*}
\|v\|_{L^\infty(K;\Real^q)} &= \|\psi_K^{-1}(\psi_K(v))\|_{L^\infty(K;\Real^q)}\le \|\polA_K^{-1}\|_{\ell^2} \|\psi_K(v)\|_{L^\infty(\wK;\Real^q)} \\ &\le c\, \|\polA_K^{-1}\|_{\ell^2} \sum_{i\in \calN}
|\wsigma_{i}(\psi_K(v))| = c\, \|\polA_K^{-1}\|_{\ell^2} \sum_{i\in \calN}|\sigma_{K,i}(v)|,
\end{align*}
using the definition of $\psi_K$, norm equivalence in $\wP$, and~\eqref{Eq2:genEF_S}.
\end{proof}

\subsection{Abstract finite element spaces}
\label{sec:broken} 
Let $\{(K,P_K,\Sigma_K)\}_{K\in\calT_h}$ be a $\calT_h$-based family
of finite elements constructed as in Proposition~\ref{Prop:GenEF}.
We consider the broken Sobolev spaces $W^{1,p}(\calT_h;\Real^q):=
\bset v\in L^p(\Dom;\Real^q)\st v_{|K}\in W^{1,p}(K;\Real^q),\ \forall
K\in \calT_h\eset$, $p\in [1,\infty]$, and 
introduce the broken finite element space
\begin{equation} \label{eq:PTh_b} 
P\upb(\calT_h) = \bset v_h\in
   L^1(\Dom;\Real^q) \tq v_{h|K} \in P_K,\, \forall K\in\calT_h\eset.
\end{equation}
Since $P_K\subset W^{1,\infty}(K;\Real^q)$, we infer
that $P\upb(\calT_h) \subset W^{1,\infty}(\calT_h;\Real^q)$.   

Our aim is to define a subspace of $P\upb(\calT_h)$ by means of some
zero-jump condition across the interfaces separating the elements.  We
say that a subset $F\subset\overline\Dom$ is an interface if it has
positive $(d{-}1)$-dimensional measure and if there are distinct mesh
cells $K_l,K_r\in\calT_h$ such that $F=\partial K_l\cap \partial K_r$.
The numbering of the two mesh cells is arbitrary, but kept fixed once
and for all, and we let $\bn_F$ be the unit normal vector to $F$
pointing from $K_l$ to $K_r$. This defines a global orientation of the
interfaces. We denote by $\bn_{K_l}$ and $\bn_{K_r}$ the outward unit
normal of $K_l$ and $K_r$, \ie $\bn_F=\bn_{K_l}=-\bn_{K_r}$.  We say
that a subset $F\subset\overline\Dom$ is a boundary face if it has
positive $(d{-}1)$-dimensional measure and if there is a mesh cell
$K\in\calT_h$ such that $F=\partial K\cap \partial \Dom$, and we let
$\bn_F$ be the unit normal vector to $F$ pointing outward $\Dom$. The
interfaces are collected in the set $\calFhi$, the boundary faces in
the set $\calFhb$, and we let $\calFh= \calFhi\cup \calFhb$.  We now
define a notion of jump across the interfaces. Recalling that
  functions in $W^{1,1}(\calT_h;\Real^q)$ have traces in
  $L^1(\partial K;\Real^q)$ for all $K\in\calT_h$, let $F\in\calFhi$
be a mesh interface, and let $K_l,K_r$ be the two cells such that
$F=\partial K_l\cap \partial K_r$; the jump of
$v\in W^{1,1}(\calT_h;\Real^q)$ across $F$ is defined to be
\begin{equation}
\label{Eq:def_jump_scal}
\jump{v}_F(\bx) =v_{|K_l}(\bx) - v_{|K_r}(\bx) \qquad \text{\ae}\ \bx \in F.
\end{equation}
In what follows, we need to consider the jump of only some components of
$v$ across $F$; we formalize this by introducing 
bounded linear operators $\gamma_{K,F}: W^{1,1}(K;\Real^q)
\to L^1(F;\Real^t)$, for all $K\in\calT_h$, all face $F\in\calFh$ that is a subset of $\partial K$, and some $t\ge1$, as follows:
\begin{equation}
\label{Eq:def_gamma_jump}
\jump{v}_F^\gamma(\bx) =\gamma_{K_l,F}(v_{|K_l})(\bx) - \gamma_{K_r,F}(v_{|K_r})(\bx) 
\qquad \text{\ae}\ \bx \in F.
\end{equation}
We assume that
$|\jump{v}_F^\gamma(\bx)|\le|\jump{v}_F(\bx)|$, \ae $\bx\in F$, for all
$v\in W^{1,1}(\calT_h;\Real^q)$. Since functions in $W^{1,1}(\Dom;\Real^q)$ have zero jump across interfaces (see, \eg \citep[Lemma~1.23]{DiPietro_Ern_dG}), we infer that
\begin{equation}
\label{eq:gamma_trace_0}
v\in W^{1,1}(\Dom;\Real^q) \; \Longrightarrow \; \jump{v}_F^\gamma =
0,\;\forall F\in \calFhi.
\end{equation}
With this setting, we introduce the 
\begin{equation}
P(\calT_h) = \bset v_h\in
 P\upb(\calT_h)  \tq \jump{v_h}_F^\gamma=0,\ \forall F\in\calFhi\eset.
\label{characterization_of_PcalTh_by_jumps}
\end{equation}

\subsection{Finite element examples} \label{Sec:Examples} The present
theory is quite general and covers a large class of scalar- and
vector-valued finite elements. For instance, it covers finite elements
of Lagrange, N\'ed\'elec, and Raviart-Thomas type. To remain general,
we denote the three reference elements corresponding to these three
classes as follows: $(\wK,\wP\upg,\wSigma\upg)$,
$(\wK,\wbP\upc,\wSigma\upc)$ and $(\wK,\wbP\upd,\wSigma\upd)$.  We
think of $(\wK,\wP\upg,\wSigma\upg)$ as a scalar-valued finite element
($q=1$) that has some degrees of freedom which require point
evaluation, 
for instance $(\wK,\wP\upg,\wSigma\upg)$ could be a Lagrange element. We
assume that the finite element $(\wK,\wbP\upc,\wSigma\upc)$ is
vector-valued ($q=d$) and some of its degrees of freedom require to
evaluate integrals over edges.
Typically, $(\wK,\wbP\upc,\wSigma\upc)$ is a
N\'ed\'elec-type or edge element.  Likewise, the finite element
$(\wK,\wbP\upd,\wSigma\upd)$ is assumed to be vector-valued ($q=d$)
and some of its degrees of freedom are assumed to require evaluation
of integrals over faces.
Typically, $(\wK,\wbP\upd,\wSigma\upd)$ is a Raviart-Thomas-type
element. It is not necessary to know the exact nature of the element
that we are handling at the moment.  We denote by $V\upg(\wK)$,
$\bV\upc(\wK)$, $\bV\upd(\wK)$ admissible domains of the degrees of
freedom in the three cases. Let $p\in [1,\infty]$. The above
assumptions imply that we can choose $V\upg(\wK) = W^{s,p}(\wK)$ with
$s>\frac{d}{p}$, $\bV\upc(\wK) = \bW^{s,p}(\wK)$ with
$s>\frac{d-1}{p}$, and $\bV\upd(\wK) = \bW^{s,p}(\wK)$ with
$s>\frac{1}{p}$.  Actually, when $p=1$ we can choose $V\upg(\wK) =
W^{d,1}(\wK)$ (since $W^{d,1}(\wK)\hookrightarrow \calC^0(\wK)$, see \eg 
\cite{Ponce_Van_Schaftingen_2007}),
$\bV\upd(\wK) = \bW^{1,1}(\wK)$ (since functions in $W^{1,1}(\wK)$
have a trace in $L^1(\partial \wK)$), and $\bV\upc(\wK) =
\bW^{d-1,1}(\wK)$ (since functions in $W^{d-1,1}(\wK)$ have traces in
$L^1$ on the one-dimensional edges of $\wK$).

Let $\calT_h$ be a mesh in the sequence $\famTh$ and let $K$ be a cell
in $\calT_h$.  We denote by $\mapKg$, $\mapKc$, $\mapKd$ the linear
maps that are classically used to generate the above
finite elements, \ie
$\mapKg$ is the pullback by $\trans_K$, and $\psi_K\upc$ and
$\psi_K\upd$ are the contravariant and covariant Piola
transformations, respectively. All of these maps fit
the general form~\eqref{localization_of_mapK}, \ie
\begin{subequations}\begin{align}
\polA_K\upg &= 1,
&& \mapKg(v) =  v\circ\trans_K, \label{Eq:Def_mapg} \\
\polA_K\upc &=\Jac_K\tr,
&&\mapKc(\bv) = \Jac_K\tr (\bv\circ\trans_K), \label{Eq:def_Piola_rot}\\
\polA_K\upd &=\det(\Jac_K)\,\Jac_K^{-1},
&& \mapKd(\bv) = \det(\Jac_K)\,\Jac_K^{-1}(\bv\circ\trans_K).\label{Eq:Def_Piola} 
\end{align}\end{subequations}
Note that $c=1$ in \eqref{Assumption:AkAkminusone_bounded} for the above examples. 

The corresponding broken finite element spaces are:
\begin{subequations}\begin{align}
P\upgb(\calT_h) &= \{v_h \in L^1(\Dom)\st 
\mapKg(v_{h|K}) \in \wP\upg, \ \forall K\in \calT_h\}, \\
\bP\upcb(\calT_h) &= \{\bv_h \in \bL^1(\Dom)\st 
\mapKc(\bv_{h|K}) \in \wbP\upc, \ \forall K\in \calT_h\},\\
\bP\updb(\calT_h) &= \{\bv_h \in \bL^1(\Dom)\st 
\mapKd(\bv_{h|K}) \in \wbP\upd, \ \forall K\in \calT_h\}.
\end{align} 
\end{subequations}
This leads us to consider the following $\gamma$-traces:
\begin{subequations}\begin{alignat}{2}
\gamma_{K,F}\upg(v_{|K})(\bx) &= v_{|K}(\bx),&\quad &\forall \bx\in F,\\
\gamma_{K,F}\upc(\bv_{|K})(\bx) &= \bv_{|K}(\bx)\CROSS \bn_F,&\quad &\forall \bx\in F,\\
\gamma_{K,F}\upd(\bv_{|K})(\bx)&= \bv_{|K}(\bx)\SCAL\bn_F,&\quad &\forall \bx\in F,
\end{alignat}\end{subequations}
and the following conforming finite element spaces:
\begin{subequations}\begin{align}
P\upg(\calT_h)   &= \{v_h \in P\upgb(\calT_h) \st \jump{v_h}_F\upg=0,\ \forall F\in\calFhi\}, \\
\bP\upc(\calT_h) &= \{\bv_h \in \bP\upcb(\calT_h) \st \jump{\bv_h}_F\upc=0,\ \forall F\in\calFhi\},\\
\bP\upd(\calT_h) &= \{\bv_h \in \bP\updb(\calT_h) \st \jump{\bv_h}_F\upd=0,\ \forall F\in\calFhi\},
\end{align} \end{subequations} where we slightly simplified the
notation by using $\jump{v_h}_F\upg$ instead of
$\jump{v_h}_F^{\gamma\upg}$, \etc Upon introducing the spaces
$V\upg := \bset v\in L^1(\Dom)\st \GRAD v\in \bL^1(\Dom)\eset$,
$\bV\upc := \bset \bv\in \bL^1(\Dom)\st \ROT \bv\in
\bL^1(\Dom)\eset$, $\bV\upd := \bset \bv\in \bL^1(\Dom)\st \DIV v\in
L^1(\Dom)\eset$, we have $P\upg(\calT_h):= P\upgb(\calT_h) \cap
V\upg$, $\bP\upc(\calT_h) := \bP\upcb(\calT_h) \cap \bV\upc$,
$\bP\upd(\calT_h) := \bP\updb(\calT_h) \cap \bV\upd$; that is to
say, the finite element spaces $P\upg(\calT_h)$, $\bP\upc(\calT_h)$,
$\bP\upd(\calT_h)$ are conforming in $V\upg$, $\bV\upc$, 
$\bV\upd$, respectively.

Let us introduce the canonical interpolation operators $\inter_h\upg$,
$\inter_h\upd$, $\inter_h\upc$ such that
$\inter_h\upg(v)_{|K} = \inter_K\upg(v_{|K})$,
$\inter_h\upc(\bv)_{|K} = \inter_K\upc(\bv_{|K})$,
$\inter_h\upd(\bv)_{|K} = \inter_K\upd(\bv_{|K})$.  The
considerations in \S\ref{Sec:Examples} show that it is legitimate to
take $W^{s,p}(\Dom)$, $s>\frac{d}{p}$, for the domain of
$\inter_h\upg$, $\bW^{s,p}(\Dom)$, $s>\frac{d-1}{p}$, for the domain of
$\inter_h\upc$, and $\bW^{s,p}(\Dom)$, $s>\frac{1}{p},$ for the domain
of $\inter_h\upd$, \ie the canonical interpolation operators
$\inter_h\upg$, $\inter_h\upc$ and $\inter_h\upd$ are not stable in
any $L^p(\Dom)$ space (or $\bL^p(\Dom)$).  The objective of this paper is to construct
quasi-interpolation operators mapping onto the spaces $P\upg(\calT_h)$,
$\bP\upc(\calT_h)$ and $\bP\upd(\calT_h)$ that are stable in
$L^1(\Dom)$ (or $\bL^1(\Dom)$) and have optimal approximation
properties with and without boundary conditions.

\subsection{Summary of the assumptions}
\label{sec:assumptions}
Let us now summarize the assumptions that will be used in the rest of
the paper.  Henceforth $\famTh$ is a shape-regular sequence of 
affine, matching meshes so that \eqref{Eq2:TransAff} and \eqref{Eq2:propJK}
hold. We also assume that the map $\mapK$
satisfies \eqref{localization_of_mapK} and
\eqref{Assumption:AkAkminusone_bounded}.
$\{(K,P_K,\Sigma_K)\}_{K\in\calT_h}$ is a $\calT_h$-based
sequence of finite elements constructed as in
Proposition~\ref{Prop:GenEF}. In view of approximation, we 
define $k$ to be the largest natural
number such that $[\polP_{k,d}]^q \subset \wP$,
where $\polP_{k,d}$ is
the real vector space of $d$-variate polynomials
functions of degree at most $k$, and we assume that
$\wP\subset W^{k+1,\infty}(\wK;\Real^q)$.

We assume that we have at hand 
a notion of $\gamma$-jump across mesh interfaces as described
in~\S\ref{sec:broken}. The 
finite element space
$P(\calT_h)$ is the subspace of the broken finite element space
$P\upb(\calT_h)$ characterized by zero $\gamma$-jumps across
interfaces, see~\eqref{characterization_of_PcalTh_by_jumps}. Finally,
two important assumptions relating the degrees of freedom to the
$\gamma$-jump and $\gamma$-trace are the
estimates~\eqref{jump_sigma_by_jump_trace} and
\eqref{Assumption_sigma_bounded_bnd} below.

In what follows, $c$ denotes a generic positive constant whose value
may depend on the shape-regularity of the mesh sequence $\famTh$ and
on the reference finite element $\wKPS$. The value of this constant
may vary from one occurrence to the other.
 
\section{$L^1$-stable local interpolation}\label{Sec:local_approximation}
In this section we extend the degrees of freedom in order to be able
to approximate functions that are only integrable.  

\subsection{Extension of degrees of freedom}
Let us consider $\wrho_i\in \wP$, $i\in\calN$, be such that
\begin{equation}
\frac{1}{\mes{\wK}}\int_\wK (\wrho_i,\wp)_{\ell^2(\Real^q)} \dif \wx  = \wsigma_i(\wp),
\quad\forall \wp\in \wP.
\end{equation}
Note that $\wrho_i$ is well defined since it is the Riesz
representative of $\wsigma_i$ in $\wP$ when $\wP$ is equipped with
the $L^2$-scalar product weighted by $1/\mes{\wK}$. This leads us to
define
\begin{equation}
\wsigma_i^\sharp(\wv) := \frac{1}{\mes{\wK}} \int_\wK (\wrho_i,\wv)_{\ell^2(\Real^q)}  \dif \wx,
\quad  \forall \wv\in L^1(\wK;\Real^q).
\end{equation}
Note that the assumption $\wP\subset L^\infty(\wK;\Real^q)$ implies that
$\wrho_i\in L^\infty(\wK;\Real^q)$, which in turn implies that all the
extended degrees of freedom $\{\wsigma_i^\sharp\}_{i\in\calN}$ are
indeed bounded over $L^1(\wK;\Real^q)$ since
$\|\wsigma_i^\sharp\|_{\calL(L^1(\wK;\Real^q);\Real)} \le
\mes{\wK}^{-1}\|\wrho_i\|_{L^\infty(\wK;\Real^q)}$.
In passing we have also proved that 
$\|\wsigma_i^\sharp\|_{\calL(L^p(\wK;\Real^q);\Real)} \le
\mes{\wK}^{-1}\|\wrho_i\|_{L^{p'}(\wK;\Real^q)}$
for all $p\in[1,\infty]$, where $\frac1p+\frac{1}{p'}=1$.  We then
define
\begin{equation}
\inter^\sharp_{\wK} (\wv) := \sum_{i\in\calN} \wsigma_i^\sharp(\wv) \wtheta_i,
\qquad  \forall \wv\in L^1(\wK;\Real^q).
\end{equation}
$\wP$ is point-wise invariant under $\inter^\sharp_{\wK}$ since
$\wsigma_i^\sharp(\wp) = \wsigma_i(\wp)$ for all $\wp\in\wP$ and all
$i\in\calN$.

Let $K\in
\calT_h$ and let $(K,P_K,\Sigma_K)$ be a finite element constructed as
in \eqref{Eq2:genEF}. 
Note that the assumption \eqref{localization_of_mapK} implies that
$\mapK(L^1(K;\Real^q)) = L^1(\wK;\Real^q)$.
We then extend the degrees of freedom in $\Sigma_K$ to $L^1(K;\Real^q)$
by setting 
\begin{equation} \label{eq:def_sigmasharp}
  \sigma_{K,i}^\sharp(v) := \wsigma_i^\sharp(\mapK(v)).
\end{equation}
The above definition leads us to define
\begin{equation} \label{eq:def_Isharp}
\inter^\sharp_{K} (v) := \sum_{i\in\calN} \sigma_{K,i}^\sharp(v) \theta_{K,i},
\qquad  \forall v\in L^1(K;\Real^q).
\end{equation}

\begin{proposition}[Stability, commutation, invariance] \label{Prop:Intersharp_bounded_commmutes_invariant}
  (i) There exists a uniform constant $c$ such that
  $\|\inter^\sharp_{K}\|_{\calL(L^p(K;\Real^q);L^p(K;\Real^q))}\le
    c$,
  for all $p\in[1,\infty]$ and all $K\in\calT_h$; (ii)  $\inter^\sharp_{K}$ commutes with $\mapK$; (iii) 
  $P_K$ is point-wise invariant under $\inter^\sharp_{K}$.
\end{proposition}
\begin{proof} 
  Using the triangle inequality in~\eqref{eq:def_Isharp}, the fact
  that $\theta_{K,i} := \mapKmun(\wtheta_i)$ and
  $\sigma_{K,i}^\sharp := \wsigma_i^\sharp \circ \mapK$
  (see~\eqref{eq:def_sigmasharp}), and finally
  using~\eqref{eq:bnd_calL_psi} with $l=0$ and the
  assumption~\eqref{Assumption:AkAkminusone_bounded}, we infer that
\begin{align*}
\|\inter^\sharp_{K}&\|_{\calL(L^p(K;\Real^q);L^p(K;\Real^q))} \\
&\le \sum_{i\in\calN} \|\wsigma_i^\sharp\|_{\calL(L^p(\wK;\Real^q);\Real)} 
  |\mapK|_{\calL(L^p(K;\Real^q);L^p(\wK;\Real^q))}\|\mapKmun(\wtheta_i)\|_{L^p(K;\Real^q)} \\
&\le |\mapK|_{\calL(L^p(K;\Real^q);L^p(\wK;\Real^q))}|\mapKmun|_{\calL(L^p(\wK;\Real^q);L^p(K;\Real^q))} \sum_{i\in\calN} \|\wsigma_i^\sharp\|_{\calL(L^p(\wK;\Real^q);\Real)} 
  \|\wtheta_i\|_{L^p(\wK;\Real^q)} \\
&\le c\, \|\polA_K\|_{\ell^2}\|\polA_K^{-1}\|_{\ell^2}
  \sum_{i\in\calN}
  \|\wsigma_i^\sharp\|_{\calL(L^p(\wK;\Real^q);\Real)} 
  \|\wtheta_i\|_{L^p(\wK;\Real^q)} \\
&\le c\,\|\polJ_K\|_{\ell^2}\|\polJ_K^{-1}\|_{\ell^2} \mes{\wK}^{-1}  \sum_{i\in\calN} \|\wrho_i\|_{L^{p'}(\wK;\Real^q)}.
\end{align*}  
The conclusion readily follows from the shape-regularity assumptions.  To
prove the second statement, we use again that
$\theta_{K,i} = \mapKmun(\wtheta_i)$ to infer that
\begin{align*}
\mapK\left(\inter^\sharp_{K} (v)\right) 
& := \mapK\left(\sum_{i\in\calN}\sigma_{K,i}^\sharp(v) \mapKmun(\wtheta_i)\right)
= \sum_{i\in\calN} \wsigma_i^\sharp(\mapK(v))\wtheta_i
= \inter^\sharp_{\wK} (\mapK(v)),
\end{align*}
for all $v\in L^1(K;\Real^q)$. To prove the third statement, let us
consider any $g\in P_K=\mapK(\wP)$; then using the above definitions we
have
\begin{align*}
\sigma_{K,i}^\sharp(g) & = \wsigma_i^\sharp(\mapK(g)) = \frac{1}{\mes{\wK}}\int_\wK 
(\wrho_i,\mapK(g))_{\ell^2(\Real^q)} \dif \wx = \wsigma_i(\mapK(g) ) = \sigma_{K,i}(g).
\end{align*}
This proves that $\inter^\sharp_K(g) = \inter_K(g)$, hence
$\inter^\sharp_K(g) = g$.
\end{proof}

\begin{remark}
  The construction of $\inter_K^\sharp$ is similar to that used in
  \cite[Appendix]{Girault_Lions_2001} to extend the Scott--Zhang quasi-interpolation
  operator and make it stable in $L^1(\Dom)$ for scalar-valued
  functions.
\end{remark}

\subsection{Error estimates for $\inter_K^\sharp$}
We establish in this section error estimates for the operator
$\inter_K ^\sharp$.

\begin{theorem}[Local interpolation]
  \label{Th:interLoc_sharp} There exists a uniform
  constant $c$ such that the following local error estimate holds:
\begin{equation}
  |v-\inter_K^\sharp v |_{W^{m,p}(K;\Real^q)} \leq c \, h_K^{l-m} |v|_{W^{l,p}(K;\Real^q)},
\label{Eq2:ineg-inter-locale}
\end{equation}
for all $l\in\intset{0}{k+1}$, all $m\in\intset{0}{l}$, 
all $p\in [1,\infty]$, all $v \in W^{l,p}(K;\Real^q)$,
and all $K\in\calT_h$. 
\end{theorem}

\begin{proof}
  Let $l\in\intset{0}{k+1}$ and $m\in\intset{0}{l}$.\\
  (1) Let us set $\calG(\ww) := \ww - \inter_{\wK}^\sharp(\ww)$ for
  all $\ww\in W^{l,p}(\wK;\Real^q)$. The operator $\calG$ is well-defined
  since $W^{l,p}(\wK;\Real^q) \hookrightarrow L^1(\wK;\Real^q)$.  Since all
  the norms are equivalent in $\wP$ and
  $\wP \subset W^{k+1,p}(\wK;\Real^q)\hookrightarrow W^{m,p}(\wK;\Real^q)$, there
  exists $c$ depending only on $\nf$ and $\wK$ such that
  $\|\inter_{\wK}^\sharp(\ww)\|_{W^{m,p}(\wK;\Real^q)}\le c
  \|\inter_{\wK}^\sharp(\ww)\|_{L^{1}(\wK;\Real^q)}$,
  which in turns implies that
  $\|\inter_{\wK}^\sharp(\ww)\|_{W^{m,p}(\wK;\Real^q)} \le c
  \|\ww\|_{L^{1}(\wK;\Real^q)}$,
  since we have already established that $\inter_{\wK}^\sharp$ is
  uniformly bounded over $L^{1}(\wK;\Real^q)$; hence,
  $\calG \in \calL(W^{l,p}(\wK;\Real^q);W^{m,p}(\wK;\Real^q))$.  Assume first that
  $l\ge1$, then $[\polP_{l-1}]^q$ is point-wise invariant under
  $\inter_{\wK}^\sharp$ since $l-1\le k$ and
  $[\polP_{l-1}]^q \subset [\polP_k]^q \subset \wP$; this in turn implies that the
  operator $\calG$ vanishes on $[\polP_{l-1}]^q$. As a consequence, we
  infer that
\begin{align*}
| \ww -  \inter_{\wK}^\sharp &\ww |_{W^{m,p}(\wK;\Real^q)} = {}| {\calG}(\ww)|_{W^{m,p}(\wK;\Real^q)}=
\inf_{\wp \in [\polP_{l-1}]^q} |{\calG}(\ww+\wp)|_{W^{m,p}(\wK;\Real^q)} \\
& \le \| {\calG} \|_{\calL(W^{l,p}(\wK;\Real^q);W^{m,p}(\wK;\Real^q))} \, \inf_{\wp \in [\polP_{l-1}]^q} \| \ww + \wp
\|_{W^{l,p}(\wK;\Real^q)}\\
& \le c \, \inf_{\wp \in [\polP_{l-1}]^q} \| \ww + \wp\|_{W^{l,p}(\wK;\Real^q)} 
\le c\, |\ww |_{W^{l,p}(\wK;\Real^q)},
\end{align*}
for all $\ww\in W^{l,p}(\wK;\Real^q)$, where the last estimate is a
consequence of the Bramble--Hilbert/Deny--Lions Lemma.
Finally, the above inequality is trivial if $l=m=0$. \\
(2) Now let $v\in W^{l,p}(K;\Real^q)$. Using the above argument together
with the fact that $\inter_K^\sharp$ commutes with $\mapK$ (see
Proposition~\ref{Prop:Intersharp_bounded_commmutes_invariant}), we
have
\[
\begin{aligned}
|v - \inter_K^\sharp  v &|_{W^{m,p}(K;\Real^q)}\le 
|\mapKmun|_{\calL(W^{m,p}(\wK;\Real^q);W^{m,p}(K;\Real^q))} \,
|\mapK(v)- \mapK(\inter_{K}^\sharp v)|_{W^{m,p}(\wK;\Real^q)}\\
 &\le 
|\mapKmun|_{\calL(W^{m,p}(\wK;\Real^q);W^{m,p}(K;\Real^q))} \,
|\mapK(v)- \inter_{\wK}^\sharp(\mapK(v))|_{W^{m,p}(\wK;\Real^q)} \\
&\le c\,
|\mapKmun|_{\calL(W^{m,p}(\wK;\Real^q);W^{m,p}(K;\Real^q))} \,
|\mapK(v)|_{W^{l,p}(\wK;\Real^q)} \\
& \le c\,
|\mapKmun|_{\calL(W^{m,p}(\wK;\Real^q);W^{m,p}(K;\Real^q))} 
|\mapK|_{\calL(W^{l,p}(K;\Real^q);W^{l,p}(\wK;\Real^q))} |v|_{W^{l,p}(K;\Real^q)}.
\end{aligned}
\]
The estimate~\eqref{Eq2:ineg-inter-locale} follows by using
\eqref{kappa_m_s_of_mapK}. \end{proof}

\section{Averaging operator}\label{Sec:averaging_operator}
In this section, we introduce a bounded linear operator
$\calJ_h\upav :P\upb(\calT_h) \to P(\calT_h)$ based on averaging.

\subsection{Connectivity array}  
Let $\{\varphi_a\}_{a\in\calA_h}$ be a basis of $P(\calT_h)$; the
functions $\varphi_a$ are called global shape functions.  We assume
that this basis is constructed so that for any $K\in \calT_h$,
either $\interior(K)\cap \supp(\varphi_{a}) =\emptyset$  or there is a unique $i\in
\calN$ such that $\varphi_{a|K} = \theta_{K,i}$.  (Recall that this is
the usual way of constructing finite element bases.)  We denote by
$\sfa: \calT_h\CROSS \calN \to \calA_h$ the map such that
$\varphi_{\sfa(K,i)|K} = \theta_{K,i}$; this map is henceforth called
connectivity array.  Note that $\sfa$ is surjective by definition, but
in general $\sfa$ is not injective. Denoting by $\sfa^{-1}(a)$ the
preimage of $a\in \calA_h$, we define the connectivity set
$\calC_a\subset\calT_h\CROSS\calN$ for any $a\in\calA_h$ such that
\begin{equation}
\calC_a := \sfa^{-1}(a) = \{(K,i)\in\calT_h\CROSS\calN \st a=\sfa(K,i)\}.
\end{equation}

\begin{remark}[Particular case
  $\card(\calC_a)=1$] \label{Rem:Card_Ta_is_one} Assume that
  $\card(\calC_a)=1$, \ie $\calC_a=\{(K_0,i_0)\}$. Let
    $K\ne K_0$, then it is not possible to find an index $i\in\calN$ such that $a=\sfa(K,i)$
    since $\card(\calC_a)=1$; this in turn implies that
    $\interior(K)\cap \supp(\varphi_{a}) =\emptyset$, hence
    $\varphi_{a|\interior(K)}=0$.  
This means
  that $\varphi_a$ is supported on one element only, \ie $\varphi_a$
  is the zero extension of $\theta_{K_0,i_0}$. Given the
  characterization of $P(\calT_h)$ assumed in
  \eqref{characterization_of_PcalTh_by_jumps}, this means that the
  $\gamma$-trace of $\varphi_a$ on the interior faces of $K$ is zero.
\end{remark}

For any $a\in \calA_h$, we set $\calF_a\upint=\emptyset$ if
$\card(\calC_a)=1$.  If $\card(\calC_a)\ge 2$ we define
$\calF_a\upint\subset\calFhi$ to be the set of the interfaces $F$ such
that there are $(K,i),(K',i')\in \calC_a$ so that $F=\partial K\cap \partial K'$.  
We now relate the $\gamma$-traces to the degrees of freedom by making
the following assumption: there exists a uniform constant $c$ such
that the following holds for all $v$ in $P\upb(\calT_h)$ and all
$a\in\calA_h$ such that $\card(\calC_a)\ge 2$:
\begin{equation}
  \left|\sigma_{K,i}(v_{|K})-\sigma_{K',i'}(v_{|K'})\right| 
  \le c\, \min(\|\polA_K\|_{\ell^2},\|\polA_{K'}\|_{\ell^2}) \|\jump{v}_F^\gamma\|_{L^\infty(F;\Real^t)},
\label{jump_sigma_by_jump_trace}
\end{equation}
for all $F\in\calF_a\upint$ and all pairs $(K,i), (K',i')\in \calC_a$
such that $F=\partial K\cap \partial K'$.  The
estimate~\eqref{jump_sigma_by_jump_trace} is natural since the
degrees of freedom of finite elements providing some conformity in
$H^1$, $\Hrt$, or $\Hdv$ are devised to coincide across
interfaces when the $\gamma$-jump is zero and when the degrees of
freedom in question belong to the same connectivity class.
In particular, the 
estimate~\eqref{jump_sigma_by_jump_trace} holds true for all the finite
  elements considered in~\S\ref{Sec:Examples}.  Owing to
\eqref{characterization_of_PcalTh_by_jumps}, the assumption
\eqref{jump_sigma_by_jump_trace} immediately implies that
\begin{equation}
\left|\sigma_{K,i}(v_{|K})-\sigma_{K',i'}(v_{|K'})\right| =0, \qquad \forall v\in P(\calT_h),\,
\forall a \in\calA_h, \,\forall (K,i), (K',i')\in \calC_a.
\label{jump_sigma_is_zero_in_PcalTh}
\end{equation}

\subsection{Averaging operator} We define the operator
$\calJ_h\upav :P\upb(\calT_h) \to P(\calT_h)$ by
\begin{equation} \label{def_Jh_upd} 
\calJ_h\upav(v) = \sum_{a\in \calA_h} \left( \frac{1}{\card(\calC_a)}\sum_{(K,i)\in \calC_a}\sigma_{K,i}(v_{|K}) \right) \varphi_a.
\end{equation}
For any $K\in\calT_h$, we introduce the notation
\begin{subequations}\begin{align}
\calT_K &:=  \cup_{i\in \calN} \bset K'\in \calT_h 
\st \exists i'\in\calN, \ (K',i') \in \calC_{\sfa(K,i)} \eset,\\
\Dom_K &:= \interior \{\cup_{K'\in \calT_K} K'\}.
\end{align}\end{subequations}
The set $\calT_K$ is the union of all the cells that share global shape
functions with $K$ and $\Dom_K$
is the interior of the collection of the points composing the cells in $\calT_K$.

\begin{lemma}[$L^p$-stability] \label{lem:stab_Oswald}
There exists a uniform constant $c$ such that
\begin{equation}
  \|\calJ_h\upav(v)\|_{L^p(K;\Real^q)}  \le c\, \| v\|_{L^p(\Dom_K;\Real^q)},
\end{equation}
for all $p\in[1,\infty]$, all $v\in P\upb(\calT_h)$, and all $K\in \calT_h$.
\end{lemma}

\begin{proof}
  We prove the result for $p=\infty$; the other cases are obtained by
  using local inverse inequalities in $P\upb(\calT_h)$.  Using the
  triangle inequality and the shape-regularity of the mesh sequence
  $\famTh$, we infer that
\begin{align*}
\|\calJ_h\upav(v)\|_{L^\infty(K;\Real^q)} 
&\le \sum_{i\in\calN} \frac{\|\theta_{K,i}\|_{L^\infty(K;\Real^q)}}{\card(\calC_{\sfa(K,i)})} \sum_{(K',i')\in \calC_{\sfa(K,i)}} \left|\sigma_{K',i'}(v_{|K'})\right|  \\
&\le c\sum_{i\in\calN} \frac{\|\polA_K^{-1}\|_{\ell^2}}{\card(\calC_{\sfa(K,i)})}\sum_{(K',i')\in \calC_{\sfa(K,i)}} \left|\sigma_{K',i'}(v_{|K'})\right| 
 \\
&\le  c \sum_{K'\in \calT_K} 
\|\polA_{K}^{-1}\|_{\ell^2}
\sum_{i'\in\calN} 
\left|\sigma_{K',i'}(v_{|K'})\right| 
\le c \,  \| v\|_{L^\infty(\Dom_K;\Real^q)},
\end{align*}
where
the last estimate results from 
$|\sigma_{K',i'}(v_{|K'})|\le |\sigma_{K',i'}(v_{|K'})-\sigma_{K,i}(v_{|K})|
+|\sigma_{K,i}(v_{|K})|$, the assumption~\eqref{jump_sigma_by_jump_trace},
the inequality $|\jump{v}_F^\gamma|_{L^\infty(F;\Real^t)} \le 
|\jump{v}_F|_{L^\infty(F;\Real^t)} \le \| v\|_{L^\infty(\Dom_K;\Real^q)}$ and~\eqref{eq:ddl_Linfty}.
\end{proof}

\begin{lemma}[Approximation by averaging] \label{lem:bnd_Osw_RT_N} There exists
  a uniform constant $c$ such that the following
  holds with the notation $\calFKi:= \cup_{i\in\calN}\calF_{\sfa(K,i)}\upint$: 
\begin{align} \label{eq:bnd_Osw_RT_N}
|v-\calJ_h\upav(v)|_{W^{m,p}(K;\Real^q)} & \le c h_K^{d\left(\tfrac1p-\tfrac1r\right) 
+ \tfrac1r-m} \sum_{F\in\calFKi} \|\jump{v}_F^\gamma\|_{L^r(F;\Real^t)}
\end{align}
for all $m\in \intset{0}{k+1}$, all $p,r\in [1,\infty]$, all $v\in P\upb(\calT_h)$, and all
$K\in\calT_h$.
\end{lemma}

\begin{proof} We only prove the bound for $m=0$ and $p=r=\infty$, the
  other cases follow by invoking standard inverse inequalities. Let
  $v\in P\upb(\calT_h)$, set $e = v-\calJ_h\upav(v)$ and
  observe that $e\in P\upb(\calT_h)$. Let $K\in\calT_h$. The bound~\eqref{eq:Linfty_ddl} implies that
  \[
\|e\|_{L^\infty(K;\Real^q)} \le c \, \|\polA_K^{-1}\|_{\ell^2}\sum_{i\in
    \calN} |\sigma_{K,i}(e_{|K})|.
\]
Owing to the definition of $\calJ_h\upav$ in \eqref{def_Jh_upd}, we first observe that
\begin{equation} \label{eq:sigma_e_Osw}
\sigma_{K,i}(e_{|K}) = \frac{1}{\card(\calC_{\sfa(K,i)})} 
\sum_{(K',i')\in\calC_{\sfa(K,i)}} \left(\sigma_{K,i}(v_{|K})-\sigma_{K',i'}(v_{|K'}) \right).
\end{equation}
Note that $\sigma_{K,i}(e_{|K})=0$ if $\card(\calC_{\sfa(K,i)})=1$
(see Remark~\ref{Rem:Card_Ta_is_one}). Let us now consider the case
$\card(\calC_{\sfa(K,i)})\ge 2$.  For all
$(K',i')\in\calC_{\sfa(K,i)}$, there is a path of mesh cells in
$\calT_K$ linking $K$ to $K'$ so that any two consecutive mesh cells
in the path share a common face $F\in\calF_{\sfa(K,i)}\upint$, and
each face is crossed only once.  Furthermore, if
$(K_l,i_l),(K_r,i_r)\in \calC_{\sfa(K,i)}$ are such that $\partial
K_l\cap \partial K_r = F\in \calF_{\sfa(K,i)}\upint$, then
\eqref{jump_sigma_by_jump_trace} implies that
\[
|\sigma_{K_l,i_l}(v_{|K_l})-\sigma_{K_r,i_r}(v_{|K_r})| \le c
\min(\|\polA_{K_l}\|_{\ell^2},\|\polA_{K_r}\|_{\ell^2})\|\jump{v}_F^\gamma\|_{L^\infty(F;\Real^t)}.
\]  
As a result, 
\[
\|e\|_{L^\infty(K;\Real^q)} \le c \, 
\|\polA_{K}^{-1}\|_{\ell^2}
\|\polA_{K}\|_{\ell^2}
\sum_{F\in \calF_{\sfa(K,i)}\upint} 
\|\jump{v}_F^\gamma\|_{L^\infty(F;\Real^t)},
\]
whence the estimate~\eqref{eq:bnd_Osw_RT_N} readily follows since
$\calFK\upint:= \cup_{i\in\calN}\calF_{\sfa(K,i)}\upint$, $\card(\calN)$ is
uniformly bounded, and the mesh sequence $\famTh$ is shape-regular.
\end{proof}

\section{Quasi-interpolation operator} \label{Sec:quasi_interpolation_operator}
Let $\inter_h^\sharp :L^1(\Dom;\Real^q)
\to P\upb(\calT_h)$ be such that
$\inter_h^\sharp(v)_{|K} = \inter_K^\sharp(v_{|K})$ for all $K\in
\calT_h$.  We now construct the global quasi-interpolation operator
$\inter_h\upav : L^1(\Dom;\Real^q) \to P(\calT_h)$ by setting
\begin{equation}
\inter_h\upav := \calJ_h\upav\circ\inter_h^\sharp.
\end{equation}
Note that $P(\calT_h)$ is point-wise invariant under $\inter_{h}\upav$
since $P(\calT_h)$ is point-wise invariant under
$\calJ_{h}\upav$ and $\inter_h^\sharp$. Hence, $\inter_{h}\upav$ is a projection, \ie
$(\inter_{h}\upav)^2=\inter_{h}\upav$.

\begin{lemma}[$W^{m,p}$-stability] \label{Lem:Lp_stability_inter_upav}
There exists a uniform constant $c$ such that 
\begin{equation}
  |\inter_h\upav(v)|_{W^{m,p}(K;\Real^q)}  \le c\, | v|_{W^{m,p}(\Dom_K;\Real^q)}, \label{eq:app.Ih.upav_RT_N}
\end{equation}
for all $p\in[1,\infty]$, all $m\in\intset{0}{k+1}$, all $v\in W^{m,p}(\Dom_K;\Real^q)$, and all $K\in \calT_h$.
\end{lemma}
\begin{proof} 
For $m=0$, the estimate follows by combining 
Proposition~\ref{Prop:Intersharp_bounded_commmutes_invariant}
with Lemma~\ref{lem:stab_Oswald}.
For $m\ge1$, the triangle inequality implies that
\[
|\inter_h\upav(v)|_{W^{m,p}(K;\Real^q)} \le
|\inter_K^\sharp(v)|_{W^{m,p}(K;\Real^q)} +
|\inter_K^\sharp(v)-\calJ_h\upav(\inter_K^\sharp(v))|_{W^{m,p}(K;\Real^q)}.
\]
Let $\term_1$ and $\term_2$ be the two terms on the right-hand side of
the above inequality. $\term_1$ is estimated using
Theorem~\ref{Th:interLoc_sharp} with $l=m$, yielding
$|\term_1| \leq c 
|v|_{W^{m,p}(K;\Real^q)}$.
$\term_2$ is estimated using Lemma~\ref{lem:bnd_Osw_RT_N} and the
fact that $v\in W^{m,p}(\Dom_K;\Real^q)\subset W^{1,1}(\Dom_K;\Real^q)$ has zero
$\gamma$-jumps across interfaces (see~\eqref{eq:gamma_trace_0}). More precisely, we have
\begin{align*}
 h_K^m |\term_2| &\le c h_K^{\tfrac1p} \sum_{F\in\calFKi} \|\jump{\inter_K^\sharp(v)}_F^\gamma\|_{L^p(F;\Real^t)} 
=  c h_K^{\tfrac1p} \sum_{F\in\calFKi} \|\jump{v-\inter_K^\sharp(v)}_F^\gamma\|_{L^p(F;\Real^t)} \\ 
&\le c  h_{K}^{\tfrac1p}\!\!\sum_{K'\in\calT_K} \sum_{F\subset \partial K'
   \cap \calFK\upint}\!\!\!\!  \|(v-\inter_K^\sharp(v))_{|K'}\|_{L^p(F;\Real^q)}
 \le c h_{K}^{m}\!\!  \sum_{K'\in\calT_K} |v|_{W^{m,p}(K';\Real^q)},
\end{align*}
where we have used the triangle inequality to bound the jump by the
values over the two adjacent mesh cells, the 
trace inequality from Lemma~\ref{Lem:trace_inequality_in_Wsp} with $s=1$,
the approximation
result of Theorem~\ref{Th:interLoc_sharp}, and the shape regularity of the mesh
sequence. Combining the bounds on $\term_1$ and $\term_2$
gives~\eqref{eq:app.Ih.upav_RT_N}.
\end{proof}

Let us now estimate the approximation properties of
  $\inter_h\upav$.  Since we are going to establish error estimates
in fractional Sobolev spaces, in the rest of the paper all the
results that are stated in fractional Sobolev spaces assume that
$p\in [1,\infty)$, whereas the results with integer degree of
  smoothness hold for $p\in [1,\infty]$.  Assuming that
$r\not\in\polN$ and denoting by $\floor{r}$ the largest integer less
than or equal to $r$, we consider the so-called Sobolev--Slobodeckij
norm defined as
$\|v\|_{W^{r,p}(\Dom;\Real^q)} = (\|v\|_{W^{\lfloor
    r\rfloor,p}(\Dom;\Real^q)}^p +
|v|_{W^{r,p}(\Dom;\Real^q)}^p)^{\frac1p}$ with
\begin{equation}
|v|_{W^{r,p}(\Dom;\Real^q)} = \left(
\sum_{|\alpha|=\lfloor r\rfloor} \int_{\Dom}\!\int_{\Dom} 
\frac{\|\partial^\alpha v(\bx)-\partial^\alpha v(\by)\|_{\ell^2(\Real^q)}^p}{\|\bx-\by\|_{\ell^2(\Real^d)}^{(r-\floor{r})p+d}}\dif x\dif y\right)^{\frac1p}.
\end{equation}

\begin{theorem}[Local approximation in fractional Sobolev spaces] \label{Thm:fractional_estimate_in_Lp}
There exists a uniform constant $c$ such that 
\begin{equation}
|v-\inter_h\upav(v)|_{W^{m,p}(K;\Real^q)} \le c \, h_{K}^{r-m}|v|_{W^{r,p}(D_K;\Real^q)}, \label{eq:est_Iav_Wrp}
\end{equation} 
for all $r\in [0,k+1]$, all
$m\in\intset{0}{\lfloor r\rfloor}$, all $p\in [1,\infty)$ if $r\not\in\polN$ or all $p\in [1,\infty]$ if $r\in\polN$, all $v\in W^{r,p}(D_K;\Real^q)$, and all
$K\in\calT_h$.
\end{theorem}
\begin{proof}
Using that $\inter_h\upav(g)=g$ for all $g\in[\polP_{k,d}]^q$ together 
with the stability of $\inter_h\upav$ in the $W^{m,p}$-seminorm
and the triangle inequality, we infer that
\begin{align*}
|v-\inter_h\upav(v)|_{W^{m,p}(K;\Real^q)} 
&= |v -g -\inter_h\upav(v-g) |_{W^{m,p}(K;\Real^q)} \\
& \le |v -g|_{W^{m,p}(K;\Real^q)}  +|\inter_h\upav(v-g) |_{W^{m,p}(K;\Real^q)} \\
&\le c\, |v -g|_{W^{m,p}(D_K;\Real^q)}.
\end{align*}
That is to say,
$|v-\inter_h\upav(v)|_{W^{m,p}(K;\Real^q)}  \le c
\inf_{g\in[\polP_{k,d}]^q} |v -g|_{W^{m,p}(D_K;\Real^q)}$.
We conclude by applying Lemma~\ref{lem:pol_app_DK} componentwise.
\end{proof}

\begin{lemma}[Polynomial approximation in $\Dom_K$]\label{lem:pol_app_DK}
The following holds:
\begin{equation} \label{eq:pol_app_DK}
\inf_{g\in\polP_{k,d}} |v -g|_{W^{m,p}(D_K)} \le
c\, h_{K}^{r-m}|v|_{W^{r,p}(D_K)},
\end{equation}
for all $r\in [0,k+1]$, all
$m\in\intset{0}{\lfloor r\rfloor}$, all $p\in [1,\infty)$ if $r\not\in\polN$ or all $p\in [1,\infty]$ if $r\in\polN$, all $v\in W^{r,p}(D_K)$, and all
$K\in\calT_h$.
\end{lemma}

\begin{proof}
We  proceed as in \cite[Thm.~1]{MR41:7819}, but instead of invoking
\cite[Thm.~3.6.11]{Morrey_1966}, where the constants may depend on $D_K$, we are going to track the
constants to make sure that they are independent of $D_K$.  
If $m=r$, there is nothing to prove. Let us assume that $m< r$. Let $\ell\in\Natural$ be such
that $\ell=r-1$ if $r$ is a natural number or $\ell=\floor{r}$
otherwise (note that $1\le r$ if $r$ is a natural number since we assumed that 
$0\le m< r$).
In both cases the integer $\ell$ is such that $m\le \ell\le k$.  Let
$\calA_{\ell,d}=\{\alpha\in\Natural^d\st
|\alpha|:=\alpha_1+\ldots+\alpha_d\le \ell\}$.
Note that
$\card(\calA_{\ell,d})=\dim(\polP_{\ell,d})={\ell+d \choose
  d}=:N_{\ell,d}$.
Since the mapping $\Phi_{\ell,d}:\polP_{\ell,d}\to \Real^{N_{\ell,d}}$
such that
$\Phi_{\ell,d}(q)=(\int_{\Dom_K} \partial^\alpha q\dif
x)_{\alpha\in\calA_{\ell,d}}$
is an isomorphism, there is a unique polynomial
$\pi_\ell(v)\in\polP_{\ell,d}$ such that
$\Phi_{\ell,d}(\pi_\ell(v))= (\int_{\Dom_K} \partial^\alpha v\dif
x)_{\alpha\in\calA_{\ell,d}}$,
\ie $\int_{\Dom_K} \partial^\alpha (v-\pi_\ell(v))\dif x=0$ for all
$\alpha\in \calA_{\ell,d}$ (this result is actually stated in
\cite[Thm.~3.6.10]{Morrey_1966}).

Since by definition
$\int_{D_K} \partial^\alpha(v-\pi_\ell(v))\dif x =0$ for all
$|\alpha|=m\le \ell$, we can apply Lemma~\ref{lem:PK_DK} below, \ie there is
a uniform constant $c$ such that
$|v-\pi_\ell(v)|_{W^{m,p}(\Dom_K)} \le c h_K
|v-\pi_\ell(v)|_{W^{m+1,p}(\Dom_K)}$.
We can repeat the argument if $m+1 \le \ell$ since in this case we
also have $\int_{D_K} \partial^\alpha(v-\pi_\ell(v))\dif x =0$ for all
$|\alpha|=m+1\le \ell$. Eventually, we obtain
\[
|v-\pi_\ell(v)|_{W^{m,p}(\Dom_K)} \le c h_K^{\ell-m} |v-\pi_\ell(v)|_{W^{\ell,p}(\Dom_K)}.
\]

If $r$ is a natural number, then $\ell+1=r$, and we can apply the above
argument one last time since
$\int_{D_K} \partial^\alpha(v-\pi_\ell(v))\dif x =0$ for all
$|\alpha|=\ell$, which then gives~\eqref{eq:pol_app_DK} because
$\partial^\alpha\pi_\ell(v) =0$ for all $|\alpha|=\ell+1$. Otherwise,
$\ell=\floor{r}$ and we apply Lemma~\ref{Lem:fracional_approx_in_LP} to all the partial
derivatives $\partial^{\alpha} (v-\pi_\ell(v))$ with $|\alpha|=\ell$,
$s=r-\floor{r}\in(0,1)$ and $O:=D_K$;
this is legitimate since all these partial derivatives have zero average
over $D_K$. We infer that there is $c$ uniform with respect to $s$,
$p$, $K$, and $v$ such that
\[
|v-\pi_\ell(v)|_{W^{m,p}(\Dom_K;\Real^q)} \le c \, h_K^{\floor{r}-m}  h_{D_K}^{r-\floor{r}}
\left(\frac{h_{D_K}^d}{\mes{D_K}}\right)^{\frac{1}{p}} |v-\pi_\ell(v)|_{W^{r,p}(D_K;\Real^q)}.
\]
Note that $|v-\pi_\ell(v)|_{W^{r,p}(D_K;\Real^q)}=|v|_{W^{r,p}(D_K;\Real^q)}$ since $\partial^{\alpha}\pi_\ell(v)$ is a constant 
in $\Real^q$ for all $|\alpha|=\ell$.
We conclude that~\eqref{eq:pol_app_DK} holds owing to the
shape-regularity of the mesh sequence.
\end{proof}

\begin{corollary}[Global best approximation in $L^p$] \label{cor:glob_best_app}
There exists a uniform constant $c$ such that 
\begin{equation} \label{Eq:Cor:interpolation_RT_N}
\inf_{w_h\in P(\calT_h)} \|v -w_h \|_{L^{p}(\Dom;\Real^q)}
\le c\, h^{r} |v|_{W^{r,p}(\Dom;\Real^q)},
\end{equation}
for all $r\in [0,k+1]$, all $p\in [1,\infty)$ if $r\not\in\polN$ or all $p\in [1,\infty]$ if $r\in\polN$, and all $v\in W^{r,p}(\Dom;\Real^q)$.
\end{corollary}

\begin{remark}[Interpolation] 
Corollary~\ref{cor:glob_best_app} can also be proved 
using~\eqref{eq:est_Iav_Wrp} for $r\in \Natural$
and the real interpolation method (\ie the $K$-method), see
\eg \cite[Chap.~22]{Tartar:07}.
\end{remark}


\begin{remark}[Approximation for $\inter_K^\sharp$] \label{Rem:fractiona_norm_inter_sharp} 
  Note in passing that Theorem~\ref{Th:interLoc_sharp}, which states the approximation properties of $\inter_K^\sharp$, can
  be re-written with fractional Sobolev norms, \ie the following
  also holds:
\begin{equation}
|v-\inter_K^\sharp(v)|_{W^{m,p}(K;\Real^q)} \le c  h_{K}^{r-m}|v|_{W^{r,p}(K;\Real^q)},
\end{equation}
for all $r\in [0,k+1]$, all
$m\in\intset{0}{\lfloor r\rfloor}$, all $p\in [1,\infty)$ if $r\not\in\polN$ or all $p\in [1,\infty]$ if $r\in\polN$, all
$v\in W^{r,p}(K;\Real^q)$, and all $K\in\calT_h$.
\end{remark}

\begin{lemma}[Poincar\'e inequality in $\Dom_K$] \label{lem:PK_DK} Let $\underline w_{\Dom_K}$ be the average of
$w$ over $\Dom_K$.
There exists a uniform constant $c$ such that
\begin{equation}
\|v-\underline v_{\Dom_K}\|_{L^p(\Dom_K)} \le c\, h_K|v|_{W^{1,p}(\Dom_K)},
\end{equation} 
for all $p\in [1,\infty]$, all $v\in W^{1,p}(\Dom_K)$, and all $K\in\calT_h$.
\end{lemma}

\begin{proof}
Let $K\in\calT_h$. Let $K_l,K_r\in\calT_K$ sharing an interface $F=\partial K_l\cap \partial K_r$. We observe that
\begin{align*}
|\underline{v}_{K_l}-\underline{v}_{K_r}| = \mes{F}^{-\frac1p} \|\underline{v}_{K_l}-v_{|K_l}+v_{|K_r}-\underline{v}_{K_r}\|_{L^p(F)},
\end{align*}
since $v_{|K_l}=v_{|K_r}$ on $F$. By using the triangle inequality,
estimating the two norms in $L^p(F)$ with the 
trace inequality~\eqref{Eq:Lem:trace_inequality_in_Wsp} (with $s=1$),
and by applying
the Poincar\'e inequality in $K_l$ and $K_r$ separately (both with
constant $\pi^{-1}$ since the mesh cells are convex sets), we  obtain
$|\underline{v}_{K_l}-\underline{v}_{K_r}| \le c
(h_{K_l}^{1-\frac1p}|v|_{W^{1,p}(K_l)} + h_{K_r}^{1-\frac1p}|v|_{W^{1,p}(K_r)})$.
After invoking the shape regularity of the mesh sequence, we infer that
\begin{equation} \label{eq:PK_saut_moyenne}
\mes{K}^{\frac1p} |\underline{v}_{K_l}-\underline{v}_{K_r}|\le c\,h_K|v|_{W^{1,p}(K_l\cup K_r)}. 
\end{equation}
Observing that $\underline{v}_{\Dom_K}-\underline{v}_{K'} =\sum_{K''\in\calT_K}
  \frac{\mes{K''}}{\mes{\Dom_K}} (\underline{v}_{K''}-\underline{v}_{K'})$
  for any $K'\in\calT_K$, we infer that
\[
\|v-\underline{v}_{\Dom_K}\|_{L^p(K')} \le \|v-\underline{v}_{K'}\|_{L^p(K')}
+ \sum_{K''\in\calT_K} \frac{\mes{K''}}{\mes{\Dom_K}} \mes{K'}^{\frac1p} |\underline{v}_{K''}-\underline{v}_{K'}|.
\]
For any $K''\in\calT_K$, we can find a path of mesh cells in $\calT_K$
linking $K'$ to $K''$ so that any consecutive mesh cells in the path
share a common face and this face is crossed only
  once. Using~\eqref{eq:PK_saut_moyenne} together with the shape
regularity of the mesh sequence, we infer that
$\|v-\underline{v}_{K'}\|_{L^p(K')} \le
ch_K|v|_{W^{1,p}(\Dom_K)}$, and the conclusion follows by summing over
$K'\in\calT_h$ and using the fact that $\card(\calT_K)$ is uniformly
bounded.
\end{proof}

\section{Quasi-interpolation with boundary
  prescription} \label{Sec:boundary_conditions} Our goal in this
section is to construct a variant of the quasi-interpolation operator
$\inter_h\upav$ that prescribes homogeneous boundary values.

\subsection{Trace operator}
Let $F\in \calFhb$ be a boundary face. We denote by $K_F$ the unique cell
such that $F\subset \partial K_F$. We consider the global trace operator $\gamma:
W^{1,1}(\Dom;\Real^q) \to L^1(\front;\Real^t)$ such that  
\begin{equation}
\gamma(v)_{|F} = \gamma_{K_F,F}(v_{|{K_F}}), \qquad \forall F\in \calFhb.
\end{equation}
We define for all $p\in[1,\infty]$
the functional space
\begin{equation}
W^{1,p}_{0,\gamma}(\Dom;\Real^q):=\{v\in W^{1,p}(\Dom;\Real^q)\tq \gamma(v)=0\}.
\end{equation}
We then set
\begin{equation}
P_0(\calT_h) := 
 \{v_h\in P(\calT_h)\tq \gamma(v_h)=0\}.
\end{equation}
The typical examples we have in mind are
\begin{subequations}\begin{align}
P\upg_0(\calT_h)   &:= \{v_h\in P\upg(\calT_h)\tq v_{h|\front}=0\},\\
\bP\upc_0(\calT_h) &:= \{\bv_h\in P\upc(\calT_h)\tq \bv_h\CROSS\bn_{|\front}=\bzero\},\\
\bP\upd_0(\calT_h) &:=\{\bv_h\in P\upd(\calT_h)\tq \bv_h\SCAL\bn_{|\front}=0\}.
\end{align} \end{subequations} Upon setting $V_0\upg = \bset v\in
V\upg\st v_{|\front} =0\eset$, $\bV_0\upc = \bset \bv \in\bV\upc\st
\bv\CROSS\bn_{|\front} =\bzero\eset$, and $\bV_0\upd = \bset \bv\in
\bV\upd \st \bv\SCAL\bn_{|\front} =0\eset$, we have $P\upg_0(\calT_h)
= P\upg(\calT_h)\cap V_0\upg$, $\bP\upc_0(\calT_h) =
\bP\upc(\calT_h)\cap \bV_0\upc$, and $\bP\upd_0(\calT_h) =
\bP\upd(\calT_h)\cap \bV_0\upd$, \ie the finite element spaces
$P\upg_0(\calT_h), \bP\upc_0(\calT_h), \bP\upd_0(\calT_h)$ are conforming in
$V_0\upg$, $\bV_0\upc$, $\bV_0\upd$, respectively.

In the rest of the paper, we slightly abuse the terminology by calling
global degrees of freedom the elements of $\calA_h$.  We say that a
global degree of freedom $a\in \calA_h$ is an internal degree of
freedom if $\gamma(\varphi_a)=0$.  The collection of all the internal
degrees of freedom is denoted $\calAhi$; the degrees of freedom in
$\calAhb =\calA_h{\setminus}\calAhi$ are called boundary degrees of
freedom.
Let $a\in\calAhb$, then there is a face $F\in\calFhb$ such that
$\gamma(\varphi_{a})_{|F}\ne 0$. Let $K_F$ be the unique cell such
that $F\subset \partial K_F$, then
$\supp(\varphi_a)\cap K_F\ne \emptyset$. This means that there is a
unique $i_{F}\in\calN$ such that
$\varphi_{a|K_F}=\theta_{K_F,i}$.  
For
all $a\in \calAhb$, we define $\calF_a^\partial$ to be the collection
of all the boundary faces $F\in\calFhb$ such that there is
$(K_F,i_F)\in \calC_a$ and $F\subset \partial K_F$; we set
$\calF_a^\partial=\emptyset$ if $a\in \calAhi$.  We abuse the notation
by setting
\begin{equation}
\jump{v}_F^\gamma(\bx) = \gamma_{K_F,F}(v_{|K_F})(\bx),\quad\text{and}
\quad \jump{v}_F(\bx) = v_{|F}(\bx),\quad \text{\ae}\ \bx\in F,\ \forall F\in \calFhb,
\end{equation}
and assume that $|\jump{v}_F^\gamma(\bx)| \le |\jump{v}_F(\bx)| $, \ae
$\bx\in F$, for all $F\in \calFhb$.  In coherence with the assumption
\eqref{jump_sigma_by_jump_trace}, we assume that there is a uniform
constant $c$ such that the following holds for all the boundary
degrees of freedom $a\in\calAhb$, all $F\in \calF_a^\partial$, all
$i_F\in\calN$ such that $(K_F,i_F)\in \calC_a$, and all
$v\in P\upb(\calT_h)$:
\begin{equation}
 |\sigma_{K_F,i_F}(v)| \le c\, \|\polA_{K_F}\|_{\ell^2} \|\gamma_{K_F,F}(v_{|K_F})\|_{L^\infty(F;\Real^t)},
\label{Assumption_sigma_bounded_bnd}
\end{equation}
  Note that
this assumption is satisfied by all the finite elements 
considered in~\S\ref{Sec:Examples}.

\subsection{Averaging and quasi-interpolation operators revisited}
We are going to modify the averaging operator $\calJ_h\upav$ to
prescribe homogeneous boundary conditions.  
We define $\calJ_{h0}\upav :
P\upb(\calT_h) \rightarrow P_{0}(\calT_h)$ by setting for all $v\in P\upb(\calT_h)$,
\begin{align} 
\calJ_{h0}\upav(v) = \sum_{a\in\calAhi} \left( 
\frac{1}{\card(\calC_{a})} \sum_{(K,i)\in \calC_a}
  \sigma_{K,i}(v_{|K}) \right) \varphi_a. \label{eq:def_Jh_upd_bnd} 
\end{align} 

\begin{lemma}[$L^p$-stability] \label{lem:stab_Oswald0}
There exists a uniform constant $c$ such that
\begin{equation}
  \|\calJ_{h0}\upav(v)\|_{L^p(K;\Real^q)}  \le c\, \| v\|_{L^p(\Dom_K;\Real^q)},
\end{equation}
for all $p\in[1,\infty]$, all $v\in P\upb(\calT_h)$, and all $K\in \calT_h$.
\end{lemma}

\begin{proof}
Proceed as in the proof of Lemma~\ref{lem:stab_Oswald}.
\end{proof}

\begin{lemma}[Approximation by averaging] \label{lem:bnd_Osw_RT_N_bnd}
There exists a uniform constant $c$ such that the
following holds: 
\begin{align} \label{eq:bnd_Osw_RT_bnd}
  |w-\calJ_{h0}\upav(w)|_{W^{m,p}(K;\Real^q)} \le c\,
  h_K^{d\left(\tfrac1p-\tfrac1r\right) + \tfrac1r-m}
  \sum_{F\in\calFKi\cup\calFKb } \|\jump{w}_F^\gamma\|_{L^r(F;\Real^t)},
\end{align} 
for all $m\in\intset{0}{k+1}$, all $p,r\in [1,\infty]$, all $w\in
P\upb(\calT_h)$, and all $K\in\calT_h$, 
with $\calFK^\partial = \cup_{i\in\calN} \calF_{\sfa(K,i)}\upbnd$ and $\calFKi$ defined in Lemma~\ref{lem:bnd_Osw_RT_N}.  
\end{lemma}
\begin{proof}
  This is a straightforward adaptation of the proof of
  Lemma~\ref{lem:bnd_Osw_RT_N}. Letting $e=v-\calJ_{h0}\upav(v)$,
  we observe that $\sigma_{K,i}(e_{|K})$ is still given by~\eqref{eq:sigma_e_Osw}
if $\sfa(K,i)\in\calAhi$, while $\sigma_{K,i}(e_{|K})=\sigma_{K,i}(w)$ if
$\sfa(K,i)\in\calAhb$, and this term is bounded using~\eqref{Assumption_sigma_bounded_bnd}.
\end{proof}

A global quasi-interpolation operator
$\inter_{h0}\upav : L^1(\Dom;\Real^q) \rightarrow P_0(\calT_h)$ is then defined
by setting
\begin{equation}
  \inter_{h0}\upav = \calJ_{h0}\upav\circ\inter_h^\sharp.
\end{equation} 
Note that $P_0(\calT_h)$ is point-wise invariant under $\inter_{h0}\upav$ since
\eqref{eq:bnd_Osw_RT_bnd} implies that $P_0(\calT_h)$ is point-wise invariant
under $\calJ_{h0}\upav$. Hence, $\inter_{h0}\upav$ is a projection,
\ie $(\inter_{h0}\upav)^2=\inter_{h0}\upav$. 
\begin{lemma}[$L^p$-stability of $\inter_{h0}\upav$] \label{Lem:Lp_stability_inter_upav0}
There is a uniform constant $c$ such that  
\begin{equation}
  \|\inter_{h0}\upav (v)\|_{L^p(K;\Real^q)}  \le c\, \| v\|_{L^p(\Dom_K;\Real^q)},
\end{equation}
for all $p\in[1,\infty]$, all $v\in L^p(\Dom;\Real^q)$, and all $K\in \calT_h$.
\end{lemma}
\begin{proof}
Proceed as for the proof of Lemma~\ref{Lem:Lp_stability_inter_upav}.
\end{proof}
%

\subsection{Error estimates}
The purpose of this section is to establish error estimates for the
quasi-interpolation operator $\inter_{h0}\upav$ in the 
${W^{r,p}}$-norm (either integer or fractional). Let $r\in [0,k+1]$ and $p\in [1,\infty)$ if $r\not\in\polN$ or $p\in [1,\infty]$ if $r\in\polN$. 
If $r>\frac1p$, then functions
in $W^{r,p}(\Dom;\Real^q)$ have traces on $\front$, and therefore it
makes sense to define
\begin{equation}
W^{r,p}_{0,\gamma}(\Dom;\Real^q):=\{v\in W^{r,p}(\Dom;\Real^q)\st\gamma(v)=0\}.
\end{equation}  
We are going to use the following notation
\begin{subequations}
\begin{align}
\calT_h\upint  &:=\{K\in\calT_h\tq \forall i\in\calN,\, \sfa(K,i)\in\calAhi\},\\
\calT_h\upbnd&:=\calT_h\setminus\calT_h\upint=\{K\in\calT_h\tq \exists i\in\calN,\, \sfa(K,i)\in\calAhb\},\\
\Dom\upbnd  &:= \interior(\cup_{K\in\calT_h\upbnd}K).
\end{align}
\end{subequations}
$\calT_h\upint$ is the set of the cells whose degrees of freedom are
all internal.  $\calT_h\upbnd$ is the set of the cells that
have at least one boundary degree of freedom.  $\Dom\upbnd$ is the
collection of the points in $\Dom$ that belong to at least one cell in
$\calT_h\upbnd$.
\begin{theorem}[Approximation]
  \label{Th:interp_fractional_s_LT_one} 
  There exists a uniform constant $c$ such that the following estimate
  holds for all $r\in [0,k+1]$, all
$m\in\intset{0}{\lfloor r\rfloor}$, all $p\in [1,\infty)$ if $r\not\in\polN$ or all $p\in [1,\infty]$ if $r\in\polN$, all $v \in W^{r,p}(\Dom_K;\Real^q)$,
and all $K\in \calT_h\upint$:
\begin{equation}
|v - \inter_{h0}\upav(v)|_{W^{m,p}(K;\Real^q)} \le c \, h_K^{r-m} |v|_{W^{r,p}(D_K;\Real^q)}.  \label{eq:Wrp_0_a}
\end{equation}
Moreover, 
\eqref{eq:Wrp_0_a} also holds if $rp>1$ with $c$ depending on $|rp-1|$
for all $v \in W^{r,p}_{0,\gamma}(\Dom;\Real^q)$ and all $K\in \calT_h\upbnd$.
Finally, if $rp<1$ (\ie $r\in(0,1)$ and $m=0$), we have
\begin{equation}
\|v - \inter_{h0}\upav(v)\|_{L^p(\Dom\upbnd;\Real^q)} \le c \, h^r \|v\|_{W^{r,p}(\Dom;\Real^q)}, \quad\forall v \in W^{r,p}(\Dom;\Real^q).
\end{equation}%
\end{theorem}
\begin{proof} Let $K$ be a cell in $\calT_h$.
If $K\in\calT_h\upint$, then 
$\inter_{h0}\upav(v)_{|K} = \inter_{h}\upav(v)_{|K}$; 
this proves~\eqref{eq:Wrp_0_a} in this case.
Let us now consider $K\in\calT_h\upbnd$.
The triangle inequality implies that
\[
|\inter_{h0}\upav(v)-v|_{W^{m,p}(K;\Real^q)} \le
|\inter_{h}\upav(v)-v|_{W^{m,p}(K;\Real^q)} +
|\inter_{h0}\upav(v)-\inter_{h}\upav(v)|_{W^{m,p}(K;\Real^q)}.
\]
Since we have already established that
$|\inter_{h}\upav(v)-v|_{W^{m,p}(K;\Real^q)} \le c h_K^{r-m}
|v|_{W^{r,p}(D_K;\Real^q)}$ in
Theorem~\ref{Thm:fractional_estimate_in_Lp}, we just need to estimate
$|\inter_{h0}\upav(v)-\inter_{h}\upav(v)|_{W^{m,p}(K;\Real^q)}$.
Let us define the set of the boundary degrees of freedom with nonempty
support on $K$, $\calA_K^\partial:=\{a\in\calAhb \st \exists i\in\calN,
\sfa(K,i)\in\calC_a\}$. Then 
\[
(\inter_{h0}\upav(v)-\inter_{h}\upav(v))_{|K} = -\sum_{a\in\calA_K^\partial}
\left(\frac{1}{\card(\calC_a)} \sum_{(K',i')\in\calC_a} \sigma_{(K',i')}(\inter_{K'}^\sharp(v|_{K'}))\right)\theta_{K,i}.
\]
For any $a\in\calA_K^\partial$ and any $(K',i')\in \calC_a$, there is a face $F\in\calF_a^\partial$ 
and a pair $(K_F,i_F)\in \calC_a$ such that there is a path of mesh cells in
$\calT_K$ linking $K'$ to $K_F$ so that any two consecutive mesh cells
in the path share a common face in $\calF_{a}\upint$, and
each face is crossed only once. This observation implies that
\begin{align*}
|\sigma_{(K',i')}(g)| & \le \sum_{F\in\calF_{a}\upint}
|\sigma_{\sfa(K_{l},i_l)}(g_{|K_{l}}) -\sigma_{\sfa(K_{r},i_r)}(g_{|K_{r}})|  + \sum_{F\in \calF_a\upbnd} |\sigma_{\sfa(K_F,i_F)}(g_{|K_F})|
\end{align*}
for all $g\in P\upb(\calT_h)$. By proceeding as in the proof of
Lemma~\ref{lem:bnd_Osw_RT_N}, we infer that
\begin{multline*}
|\inter_{h0}\upav(v)-\inter_{h}\upav(v)|_{W^{m,p}(K;\Real^q)} \le c h_K^{-m+\frac1p}
  \sum_{a\in\calA_K^\partial} 
\Bigg[\sum_{F\in\calF_a\upint}
  \|\jump{\inter^\sharp(v)}_F^\gamma\|_{L^p(F;\Real^t)} \\
+ \sum_{F\in\calF_a^\partial}
 \|\gamma(\inter^\sharp(v))_{|F}\|_{L^p(F;\Real^t)} \Bigg].
\end{multline*}

\underline{Case 1, $rp>1$:} Let us assume that $v\in
W^{r,p}_{0,\gamma}(\Dom;\Real^q)$. The boundary condition
$\gamma(v)_{F}=0$, for all $F\in\calFhb$, and the continuity condition
  $\jump{v}_F^\gamma=0$, for all $F\in\calF_h\upint$, imply that
\begin{align*}
  |\inter_{h0}\upav(v)-\inter_{h}\upav(v)|_{W^{m,p}(K;\Real^q)}
  &\le c\, h_K^{-m+\tfrac1p} \sum_{F\in\calFK\upint\cup\calFK\upbnd} 
\|\inter_K^\sharp(v)
  -v\|_{L^p(F;\Real^t)}.
\end{align*}
The conclusion follows by invoking the trace inequality from
Lemma~\ref{Lem:trace_inequality_in_Wsp}, either with $s=1$ if
$r\in\Natural$ or with $s=r-\lfloor r\rfloor$ otherwise, and
the approximation properties of $\inter_K^\sharp$ stated either in
  Theorem~\ref{Th:interLoc_sharp} or in
Remark~\ref{Rem:fractiona_norm_inter_sharp}.

\underline{Case 2, $rp<1$:} 
Assume now that $rp<1$. Norm equivalence implies that
\begin{align*}
\|\inter_{h0}\upav(v)-\inter_{h}\upav(v)\|_{L^p(\Dom\upbnd;\Real^q)}^p
\le c\, \|\inter_h^\sharp(v))\|_{L^p(\Dom\upbnd;\Real^q)}^p.
\end{align*}
Let $\rho$ be the distance to $\front$;
then there is $c$ uniform with respect to the mesh sequence such that
$\|\rho\|_{L^\infty(\Dom\upbnd)}\le c h$ and 
\begin{align*}
\|\inter_{h0}\upav(v)-\inter_{h}\upav(v)\|_{L^p(\Dom\upbnd;\Real^q)} 
& \le c \big( \|\inter_h^\sharp(v) -v\|_{L^p(\Dom\upbnd;\Real^q)} 
+ \|v\|_{L^p(\Dom\upbnd;\Real^q)} \big)\\
&\le c \left(h^r \|v\|_{W^{r,p}(\Dom;\Real^q)} + \|\rho^r \rho^{-r}v\|_{L^p(\Dom\upbnd;\Real^q)} \right)\\
&\le c \left(h^r \|v\|_{W^{r,p}(\Dom;\Real^q)} + \|\rho\|_{L^\infty(\Dom\upbnd)}^r \|\rho^{-r}v\|_{L^p(\Dom\upbnd;\Real^q)} \right).
\end{align*}
 Since $rp<1$, we infer that  (see
\cite[Cor.~1.4.4.5]{Grisvard_85})
\[
\|\rho^{-r}v\|_{L^p(\Dom\upbnd;\Real^q)}\le\|\rho^{-r}v\|_{L^p(\Dom;\Real^q)} \le c \|v\|_{W^{r,p}(\Dom;\Real^q)}.
\]
In conclusion,
$\|\inter_{h0}\upav(v)-\inter_{h}\upav(v)\|_{L^p(\Dom\upbnd;\Real^q)} \le c
h^r \|v\|_{W^{r,p}(\Dom;\Real^q)}$.
\end{proof}

\begin{corollary}[Global best approximation in $L^p$]
There exists a uniform constant $c$, additionally depending on $|rp-1|$, such that
\begin{equation}
\!\,\inf_{w_h \in P_0(\calT_h)}\!\|v-w_h\|_{L^p(\Dom;\Real^q)} \! \le 
\begin{cases}
c h^r |v|_{W^{r,p}(\Dom;\Real^q)},& \text{$\forall v\in W^{r,p}_{0,\gamma}(\Dom;\Real^q)$ if $rp>1$}\\
c h^r \|v\|_{W^{r,p}(\Dom;\Real^q)},& \text{$\forall v\in W^{r,p}(\Dom;\Real^q)$ if $rp<1$}.
\end{cases}\hspace{-0.5cm}
\end{equation}
\end{corollary}

\begin{remark}[Theorem~\ref{Th:interp_fractional_s_LT_one}]
For $rp>1$, a similar estimate has been obtained in the scalar-valued case for 
the Scott--Zhang interpolation operator by~\cite{Ciarlet:13}. 
Furthermore, the estimate for $rp<1$ in Theorem~\ref{Th:interp_fractional_s_LT_one} essentially says
that the difference $v - \inter_{h0}\upav(v)$ does not blow up
too fast close to the boundary. A
better result is not expected since $\inter_{h0}\upav(v)$ is forced to
be zero at $\front$ whereas $v$ can blow up like $\rho^{-s} w$ 
where $w$ is a function in $L^p(\Dom;\Real^q)$.
\end{remark}

\begin{remark}[$rp=1$]
Let $r\in (0,1)$. Using the notation from  the real interpolation theory,
it is known that $W^{r,p}(\Dom) = [L^p(\Dom),W^{1,p}(\Dom)]_{r,p}$ since $\Dom$ is Lipschitz,
see \cite[Lem.~36.1]{Tartar:07}. Let us define
\begin{equation}
W^{r,p}_{00,\gamma}(\Dom;\Real^q) := [L^p(\Dom;\Real^q),W^{1,p}_{0,\gamma}(\Dom;\Real^q) ]_{r,p}.
\end{equation}
Then, using Theorem~\ref{Th:interp_fractional_s_LT_one} with
$l\in\{0,1\}$ and $m=0$, the real interpolation theory implies
that
\begin{equation} 
\|v-\inter_{h0}\upav(v)\|_{L^p(\Dom;\R^q)}\le 
c\, h^r\|v\|_{W^{r,p}_{00,\gamma}(\Dom;\Real^q)},\label{Error_Interh0_Interpolation}
\end{equation}
for all $p\in [1,\infty)$ and all
$v\in W^{r,p}_{00,\gamma}(\Dom;\Real^q)$.  This estimate is not fully
satisfactory for two reasons.  First it is not local. Second it is not
really clear what $W^{\frac1p,p}_{00,\gamma}(\Dom;\Real^q)$ is.  For
instance, let us define
\begin{subequations}\begin{align}
W^{1,p}_0(\Dom)&:=\{v\in W^{1,p}(\Dom)\st v_{|\front}=0\},\\
\bW^{1,p}_T(\Dom)&:=\{\bv\in \bW^{1,p}(\Dom)\st\bv\CROSS\bn_{|\front}=\bzero\},\\
\bW^{1,p}_N(\Dom)&:=\{\bv\in \bW^{1,p}(\Dom)\st\bv\SCAL\bn_{|\front}=0\}.
\end{align}\end{subequations}
One then realizes that characterizing
$[\bL^p(\Dom),\bW^{1,p}_T(\Dom)]_{\frac1p,p}$ and
$[\bL^p(\Dom),\bW^{1,p}_N(\Dom)]_{\frac1p,p}$ in terms of Sobolev regularity
is (possible but) not straightforward, and to the best of our
knowledge, a full characterization of
these spaces is not yet available.
\end{remark}

\section{Technical results in fractional Sobolev spaces}
\label{sec:technical}

This section contains two technical results in fractional Sobolev spaces:
a Poincar\'e inequality and a trace inequality.

\begin{lemma}[Poincar\'e inequality] \label{Lem:fracional_approx_in_LP}
Let $O$ be an open set in $\Real^d$ and let $\underline{v}_O$ be the average of $v$ over
$O$, for any $v\in L^1(O)$. Let $h_O:=\diam(O)$. Then, for all $v\in
W^{s,p}(O)$ with $s\in(0,1)$ and $p\in [1,\infty)$, the following holds:
\begin{equation}
\|v -\underline{v}_{O}\|_{L^p(O)} \le h_O^{s}
\left(\frac{h_O^d}{\mes{O}}\right)^{\frac{1}{p}} |v|_{W^{s,p}(O)}. \label{Wsp_Poincare}
\end{equation}
\end{lemma}
\begin{proof}
  This result is essentially Proposition~6.1 from
  \cite{Dupont_Scott_1980}, see also~\cite{Heuer:14}.  We nevertheless give a proof since the
  computation gives us the constant in the right-hand side of \eqref{Wsp_Poincare},
  and this in turn allows us to apply the result locally on shape-regular
  meshes. Using the definitions, we have
\begin{align*}
\int_{O} |v(\bx) - \underline{v}_{O}|^p\!\dif x 
&= \int_{O} \mes{O}^{-p}\left|\int_{O} (v(\bx) - v(\by))\dif y\right|^p\!\!\!\dif x \\
&\le \int_O \mes{O}^{-p}\left(\int_{O} 
\frac{|v(\bx) - v(\by)|}{\|\bx-\by\|_{\ell^2}^{s+\frac{d}{p}}}\|\bx-\by\|_{\ell^2}^{s+\frac{d}{p}}
\dif y\right)^p\dif x  \\
& \le \int_O \mes{O}^{-p} \int_{O} 
\frac{|v(\bx) - v(\by)|^p}{\|\bx-\by\|_{\ell^2}^{sp+d}} \dif y
\left(\int_O\|\bx-\by\|_{\ell^2}^{(s+\frac{d}{p})p'}\dif y\right)^\frac{p}{p'}\!\!\!\dif x,
\end{align*}
where $p':=\frac{p}{p-1}$. Then  using that
$\|\bx-\by\|_{\ell^2}\le h_O$ for all $\bx,\by\in O$, we infer that
\begin{align*}
\|v - \underline{v}_{O}\|_{L^p(O)}^p
&\le \int_O \mes{O}^{-p} \int_O 
\frac{|v(\bx) - v(\by)|^p}{\|\bx-\by\|_{\ell^2}^{sp+d}} \dif y\dif x
\left(\max_{\bx\in O}\int_O\|\bx-\by\|_{\ell^2}^{(s+\frac{d}{p})p'}\dif y\right)^\frac{p}{p'}\\
& \le |v|_{W^{s,p}(O)}^p \mes{O}^{-p} 
\left(\int_O h_O^{(s+\frac{d}{p})p'}\dif y\right)^\frac{p}{p'}\\
&\le  |v|_{W^{s,p}(O)}^p  \mes{O}^{-p} \mes{O}^{\frac{p}{p'}} 
h_O^{sp+d}  \le  |v|_{W^{s,p}(O)}^p h_O^{sp+d} \mes{O}^{-1}.
\end{align*}
Hence $\|v - \underline{v}_{O}\|_{L^p(O)} \le h_O^{s}
\left(\frac{h_O^d}{\mes{O}}\right)^{\frac{1}{p}} |v|_{W^{s,p}(O)}$.
\end{proof}

\begin{lemma}[Trace inequality] 
  \label{Lem:trace_inequality_in_Wsp} Assume $s\in(0,1)$ and
    $sp>1$ with $p\in[1,\infty)$ or $s=1$ with $p\in [1,\infty]$.  There exists $c$, uniform with
  respect to the mesh sequence but depending on $|sp-1|$ if $s\in (0,1)$, such that
  the following holds for all $v\in W^{s,p}(K)$ and all
  $K\in\calT_h$: \begin{equation} \|v\|_{L^p(F)} \le
    c(h_K^{-\frac{1}{p}}\|v\|_{L^p(K)} +
    h_K^{s-\frac{1}{p}}|v|_{W^{s,p}(K)}). \label{Eq:Lem:trace_inequality_in_Wsp}
\end{equation} 
\end{lemma}
\begin{proof}
  Let $v\in W^{s,p}(\Dom)$. Let
  $K\in \calT_h$ be a cell and $F$ be a face of $K$.  
  Since the map $\trans_K$ is affine,
using a trace inequality in $W^{s,p}(\wK)$ (recall that $s\in(0,1)$ and
    $sp>1$ or $s=1$ and $p\ge 1$),
we infer that
\[
\|v\|_{L^p(F)} = \frac{\mes{F}^{\frac1p}}{\mes{\wF}^{\frac1p}}\|\psi_K\upg(v)\|_{L^p(\wF)}
\le c_{s,p} \mes{F}^{\frac1p}(\|\psi_K\upg(v)\|_{L^p(\wK)} + |\psi_K\upg(v)|_{W^{s,p}(\wK)}),
\]
where $c_{s,p}$ depends on $|sp-1|$ if $s\in (0,1)$.
Upon changing variables, this inequality is re-written
\[
\|v\|_{L^p(F)} 
\le c_{s,p} \mes{F}^{\frac{1}{p}} \mes{K}^{-\frac1p}(\|v\|_{L^p(K)} + \|\polJ_K\|_{\ell^2}^{-s} |v|_{W^{s,p}(K)}).
\]
The conclusion follows from the shape-regularity of the mesh sequence,
see \eqref{Eq2:propJK}.
\end{proof}

\bibliographystyle{abbrvnat} 
\bibliography{ref_quasi_interp}

\begin{thebibliography}{24}
\providecommand{\natexlab}[1]{#1}
\providecommand{\url}[1]{\texttt{#1}}
\expandafter\ifx\csname urlstyle\endcsname\relax
  \providecommand{\doi}[1]{doi: #1}\else
  \providecommand{\doi}{doi: \begingroup \urlstyle{rm}\Url}\fi

\bibitem[Achdou et~al.(2003)Achdou, Bernardi, and Coquel]{AcBeC:03}
Y.~Achdou, C.~Bernardi, and F.~Coquel.
\newblock A priori and a posteriori analysis of finite volume discretizations
  of {D}arcy's equations.
\newblock \emph{Numer. Math.}, 96\penalty0 (1):\penalty0 17--42, 2003.

\bibitem[Arnold et~al.(2006)Arnold, Falk, and
  Winther]{Arnold_Falk_Whinter_2006}
D.~N. Arnold, R.~S. Falk, and R.~Winther.
\newblock Finite element exterior calculus, homological techniques, and
  applications.
\newblock \emph{Acta Numer.}, 15:\penalty0 1--155, 2006.

\bibitem[Bernardi and Girault(1998)]{BerGi:98}
C.~Bernardi and V.~Girault.
\newblock A local regularization operator for triangular and quadrilateral
  finite elements.
\newblock \emph{SIAM J. Numer. Anal.}, 35\penalty0 (5):\penalty0 1893--1916
  (electronic), 1998.

\bibitem[Bramble and Hilbert(1970)]{MR41:7819}
J.~H. Bramble and S.~R. Hilbert.
\newblock Estimation of linear functionals on {S}obolev spaces with application
  to {F}ourier transforms and spline interpolation.
\newblock \emph{SIAM J. Numer. Anal.}, 7:\penalty0 112--124, 1970.

\bibitem[Burman and Ern(2007)]{BurEr:07}
E.~Burman and A.~Ern.
\newblock Continuous interior penalty {$hp$}-finite element methods for
  advection and advection-diffusion equations.
\newblock \emph{Math. Comp.}, 76\penalty0 (259):\penalty0 1119--1140, 2007.

\bibitem[Campos~Pinto and Sonnendr{\"u}cker(2015)]{CamSo:15}
M.~Campos~Pinto and E.~Sonnendr{\"u}cker.
\newblock Gauss-compatible {G}alerkin schemes for time-dependent {M}axwell
  equations.
\newblock \emph{Math. Comp.}, 2015.
\newblock to appear.

\bibitem[Ciarlet(2013)]{Ciarlet:13}
P.~Ciarlet, Jr.
\newblock Analysis of the {S}cott-{Z}hang interpolation in the fractional order
  {S}obolev spaces.
\newblock \emph{J. Numer. Math.}, 21\penalty0 (3):\penalty0 173--180, 2013.

\bibitem[Ciarlet(2002)]{Ciarlet_FE_Book_2002}
P.~G. Ciarlet.
\newblock \emph{The finite element method for elliptic problems}, volume~40 of
  \emph{Classics in Applied Mathematics}.
\newblock Society for Industrial and Applied Mathematics (SIAM), Philadelphia,
  PA, 2002.
\newblock Reprint of the 1978 original [North-Holland, Amsterdam; MR0520174 (58
  \#25001)].

\bibitem[Cl{\'e}ment(1975)]{Cl75}
P.~Cl{\'e}ment.
\newblock Approximation by finite element functions using local regularization.
\newblock \emph{RAIRO, Anal. Num.}, 9:\penalty0 77--84, 1975.

\bibitem[Cockburn et~al.(2007)Cockburn, Kanschat, and Sch{\"o}tzau]{CoKaS:07}
B.~Cockburn, G.~Kanschat, and D.~Sch{\"o}tzau.
\newblock A note on discontinuous {G}alerkin divergence-free solutions of the
  {N}avier-{S}tokes equations.
\newblock \emph{J. Sci. Comput.}, 31\penalty0 (1-2):\penalty0 61--73, 2007.
\newblock ISSN 0885-7474.

\bibitem[Di~Pietro and Ern(2012)]{DiPietro_Ern_dG}
D.~Di~Pietro and A.~Ern.
\newblock \emph{Mathematical aspects of discontinuous {G}alerkin methods},
  volume~69 of \emph{Math\'ematiques \& Applications (Berlin) [Mathematics \&
  Applications]}.
\newblock Springer, Heidelberg, 2012.

\bibitem[Dupont and Scott(1980)]{Dupont_Scott_1980}
T.~Dupont and R.~Scott.
\newblock Polynomial approximation of functions in {S}obolev spaces.
\newblock \emph{Math. Comp.}, 34\penalty0 (150):\penalty0 441--463, 1980.

\bibitem[Ern and Guermond(2004)]{ErnGuermond_FEM}
A.~Ern and J.-L. Guermond.
\newblock \emph{Theory and practice of finite elements}, volume 159 of
  \emph{Applied Mathematical Sciences}.
\newblock Springer-Verlag, New York, 2004.

\bibitem[Ern et~al.(2010)Ern, Stephansen, and Vohral{\'{\i}}k]{ErnSV:10}
A.~Ern, A.~F. Stephansen, and M.~Vohral{\'{\i}}k.
\newblock Guaranteed and robust discontinuous {G}alerkin a posteriori error
  estimates for convection-diffusion-reaction problems.
\newblock \emph{J. Comput. Appl. Math.}, 234\penalty0 (1):\penalty0 114--130,
  2010.

\bibitem[Girault and Lions(2001)]{Girault_Lions_2001}
V.~Girault and J.-L. Lions.
\newblock Two-grid finite-element schemes for the transient {N}avier-{S}tokes
  problem.
\newblock \emph{M2AN Math. Model. Numer. Anal.}, 35\penalty0 (5):\penalty0
  945--980, 2001.

\bibitem[Grisvard(1985)]{Grisvard_85}
P.~Grisvard.
\newblock \emph{Elliptic problems in nonsmooth domains}, volume~24 of
  \emph{Monographs and Studies in Mathematics}.
\newblock Pitman (Advanced Publishing Program), Boston, MA, 1985.

\bibitem[Heuer(2014)]{Heuer:14}
N.~Heuer.
\newblock On the equivalence of fractional-order {S}obolev semi-norms.
\newblock \emph{J. Math. Anal. Appl.}, 417\penalty0 (2):\penalty0 505--518,
  2014.

\bibitem[Karakashian and Pascal(2003)]{KarPa:03}
O.~A. Karakashian and F.~Pascal.
\newblock A posteriori error estimates for a discontinuous {G}alerkin
  approximation of second-order elliptic problems.
\newblock \emph{SIAM J. Numer. Anal.}, 41\penalty0 (6):\penalty0 2374--2399,
  2003.

\bibitem[Morrey(1966)]{Morrey_1966}
C.~B. Morrey, Jr.
\newblock \emph{Multiple integrals in the calculus of variations}.
\newblock Die Grundlehren der mathematischen Wissenschaften, Band 130.
  Springer-Verlag New York, Inc., New York, 1966.

\bibitem[Oswald(1993)]{Oswald_Comput_1993}
P.~Oswald.
\newblock On a {BPX}-preconditioner for {${\rm P}1$} elements.
\newblock \emph{Computing}, 51\penalty0 (2):\penalty0 125--133, 1993.

\bibitem[Ponce and Van~Schaftingen(2007)]{Ponce_Van_Schaftingen_2007}
A.~C. Ponce and J.~Van~Schaftingen.
\newblock The continuity of functions with {$N$}-th derivative measure.
\newblock \emph{Houston J. Math.}, 33\penalty0 (3):\penalty0 927--939, 2007.

\bibitem[Sch{\"o}berl and Lehrenfeld(2013)]{Schoberl_Lehren:13}
J.~Sch{\"o}berl and C.~Lehrenfeld.
\newblock Domain decomposition preconditioning for high order hybrid
  discontinuous {G}alerkin methods on tetrahedral meshes.
\newblock In \emph{Advanced finite element methods and applications}, volume~66
  of \emph{Lect. Notes Appl. Comput. Mech.}, pages 27--56. Springer,
  Heidelberg, 2013.

\bibitem[Scott and Zhang(1990)]{ScoZh:90}
R.~L. Scott and S.~Zhang.
\newblock Finite element interpolation of nonsmooth functions satisfying
  boundary conditions.
\newblock \emph{Math. Comp.}, 54\penalty0 (190):\penalty0 483--493, 1990.

\bibitem[Tartar(2007)]{Tartar:07}
L.~Tartar.
\newblock \emph{An introduction to {S}obolev spaces and interpolation spaces},
  volume~3 of \emph{Lecture Notes of the Unione Matematica Italiana}.
\newblock Springer, Berlin; UMI, Bologna, 2007.

\end{thebibliography}

\end{document}